\documentclass{scrartcl}
\usepackage[numbers]{natbib}

\usepackage[dvipsnames]{xcolor} 

\usepackage[bookmarks=true]{hyperref}

\usepackage{amsmath}
\usepackage{amsthm}
\usepackage{amssymb}
\usepackage{stmaryrd}
\usepackage{mathtools}
\usepackage{bm}

\usepackage{glossaries-extra}

\usepackage{tikz} 
\usepackage{tikz-cd}

\theoremstyle{plain}
\newtheorem{theorem}{Theorem}[section]
\newtheorem{lemma}[theorem]{Lemma}
\newtheorem{proposition}[theorem]{Proposition}
\newtheorem{corollary}[theorem]{Corollary}

\theoremstyle{definition}
\newtheorem{definition}[theorem]{Definition}
\newtheorem{example}[theorem]{Example}
\newtheorem{remark}[theorem]{Remark}

\newcommand{\A}{\mathbb{A}}
\newcommand{\C}{\mathbb{C}}
\newcommand{\F}{\mathbb{F}}
\newcommand{\G}{\mathbb{G}}

\newcommand{\PP}{\mathbb{P}}

\newcommand{\R}{\mathbb{R}}

\newcommand{\Z}{\mathbb{Z}}

\newcommand{\frakm}{\mathfrak{m}}

\newcommand{\calA}{\mathcal{A}}
\newcommand{\calB}{\mathcal{B}}

\newcommand{\calV}{\mathcal{V}}

\newcommand{\Spec}{\mathrm{Spec}}

\newcommand{\Tr}{\mathrm{Tr}}

\newcommand{\Hom}{\mathrm{Hom}}
\newcommand{\Gal}{\mathrm{Gal}}

\newcommand{\pr}{\mathrm{pr}}
\newcommand{\ext}{\mathrm{ext}}

\DeclareMathOperator{\NP}{\mathrm{NP}}
\DeclareMathOperator{\HP}{\mathrm{HP}}
\DeclareMathOperator{\cHP}{\mathrm{cHP}}

\newcommand{\tr}{\mathrm{tr}}
\newcommand{\D}{\mathrm{d}}

\newcommand{\Frob}{\mathrm{Frob}}

\mathchardef\mhyphen="2D

\DeclareOldFontCommand{\sf}{\normalfont\sffamily}{\mathsf}

\title{Newton Polygons of Sums on Curves I: Local-to-Global Theorems}
\author{Joe Kramer-Miller and James Upton}

\date{\today}

\begin{document}

\maketitle

\begin{abstract}	
	The purpose of this article is to study
	Newton polygons of certain abelian $L$-functions
	on curves. Let $X$ be a smooth affine curve over a finite field $\F_q$
	and let $\rho:\pi_1(X) \to \C_p^\times$ be a finite character of order $p^n$.
	By previous work of the first author, the
	Newton polygon $\NP(\rho)$ lies
	above a `Hodge polygon' $\HP(\rho)$, which is defined using local ramification invariants
	of $\rho$. In this article we study the touching between these two polygons.
	We prove that $\NP(\rho)$ and $\HP(\rho)$ share a vertex if and only
	if a corresponding vertex is shared between the Newton and Hodge polygons of
	`local' $L$-functions associated to each ramified point of $\rho$. 
	As a consequence, we determine a necessary and sufficient condition for the coincidence of $\NP(\rho)$ and $\HP(\rho)$. 
	
\end{abstract}

\tableofcontents

\section{Introduction}
\subsection{Motivation}

Let $p$ be an odd prime and let $q$ be a power of $p$. Let $X/\F_q$ be a smooth affine variety with \'etale fundamental group $\pi_1(X)$. Consider a continuous character $\rho: \pi_1(X) \to \C^\times$ of the \'etale fundamental group of $X$.
A fundamental problem in arithmetic geometry is to understand the character sums
\begin{align*}
S_k(X,\rho)= \sum_{x \in X(\mathbb{F}_{q^k})} \rho(\text{Frob}_x),
\end{align*}
where $\text{Frob}_x$ denotes the Frobenius element at $x$. By the Weil conjectures, there exist algebraic integers $\alpha_1,\dots,\alpha_{d_1}, \beta_1,\dots,\beta_{d_2}$
such that
\begin{align*}
S_k(X,\rho) &= \sum_{i=1}^{d_1} \beta_i^k - \sum_{i=1}^{d_2} \alpha_i^k.
\end{align*}
What can we say about these algebraic integers? The Riemann Hypothesis states
that the $\alpha_i$ and $\beta_i$ are $q$-Weil numbers, i.e. that there exist positive integers 
$a_1,\dots,a_{d_1}, b_1,\dots,b_{d_2}$ such that for any Archemedian place $|\cdot |_\infty$
we have $|\alpha_i|_\infty=q^{\frac{a_i}{2}}$ and $|\beta_i|_\infty=q^{\frac{b_i}{2}}$. 
We also understand what happens at the finite places away from $p$:
both the $\alpha_i$ and the $\beta_i$ are $\ell$-adic units for any prime $\ell\neq p$.

The $p$-adic properties of these numbers are less definitive.
There is no exact formula as in the $\ell$-adic and Archimedean cases.
In general, the best one can hope for are lower bounds on the $p$-adic valuations, i.e.,
that the Newton polygon of $\rho$ lies above some sort of Hodge polygon. Under additional 
congruence conditions involving $p$ and the local monodromy of $\rho$, one may hope to completely determine the $p$-adic valuations (i.e. that the Newton and Hodge
polygons coincide). The classical approach to `Newton over Hodge' results for exponential sums has been restricted to the
case where $X$ is an algebraic torus. This wonderful theory was developed by Robba, Adolphson-Sperber, Wan,
and others, building off of ideas of Dwork.
For instance, Adolphson-Sperber prove a general `Newton over Hodge' result for exponential sums over $\mathbb{G}_m^{d}$
(see \cite{Adolphson-Sperber}) and Wan establishes a criterion for equality of the two polygons (see \cite{Wan4}).
Outside of the $\mathbb{G}_m^d$ case nothing was known until recent work of the first author (see \cite{Kramer-Miller3}
and \cite{KramerMiller}). In these articles Kramer-Miller establishes a `Newton over Hodge' theorem for 
abelian $L$-functions on curves of any genus. Unfortunately, this work ignores the question of touching or equality
between these two polygons.

The purpose of this article is to study touching between the Newton and Hodge polygons
of character sums on higher genus curves. Our main result establishes a local criterion that is necessary and sufficient for the Newton and Hodge polygons to touch at a vertex. It roughly states the following: Assume the compactification of $X$ is an ordinary curve $\overline{X}$.\footnote{By the Deuring-Shafarevich formula, if $\overline{X}$ is non-ordinary, we immediately know that the Newton polygon will be strictly above the Hodge polygon} For each point at infinity $P \in \overline{X} \backslash X$, we may localize $\rho$ at $P$ and
construct a local-to-global extension $\rho_{P}^\ext$, in the sense of Gabber-Katz. This
$\rho_P^\ext$ is a representation of $\pi_1(\mathbb{G}_m)$, whose localization at $\infty$ agrees with the localization of $\rho$ at $P$. When $\rho$ has order $p^n$, we prove that the Newton and Hodge polygons of $\rho$ share a vertex if and only if the Newton and Hodge polygons of $\rho_P^\ext$ share a corresponding vertex for every $P \in \overline{X}\backslash X$. This
allows us to reduce the study of Newton-Hodge interaction to the classical $\mathbb{G}_m$ case. In particular, we establish a necessary and sufficient local condition for the Newton and Hodge polygons of $\rho$ to agree. In subsequent papers
we will apply these local-to-global theorems to study $\Z_p$-towers of curves. 

\subsection{Main results}\label{ss:mainresults}

Suppose now that $X$ is a smooth affine curve over $\F_q$ with smooth compactification $\overline{X}$. Let $S$ denote the complement of $X$ in $\overline{X}$. Let $\rho:\pi_1(X) \to \C^\times_p$ be a character of order $p^n$, assumed to be totally ramified at each $P \in S$. We will write
\begin{equation*}
	L(\rho,s) =	\prod_{x \in |X|} \frac{1}{1-\rho(\Frob_x)s^{\deg(x)}}	\in 1+s\Z_p[\rho][s]
\end{equation*}
for the Artin $L$-function of $\rho$ and $\NP_q(\rho)$ for the $q$-adic Newton polygon of $L(\rho,s)$.

For each $P \in S$, let $F_P$ denote the local field at $P$ and let $G_P$ denote the absolute Galois group of $F_P$. We localize $\rho$ to obtain a continuous character
\begin{equation*}
	\rho_P:G_P \to \C^\times.
\end{equation*}
By a theorem of Katz-Gabber, the local character $\rho_P$ extends to a character $\rho_P^\ext$ of the fundamental group $\pi_1(\A_{\F_q}^1)$. We may therefore speak of the \emph{local Newton polygon} $\NP_q(\rho_P^\ext)$ of $\rho$ at $P$. If $d_P$ denotes the Swan conductor of $\rho$ at $P$, then we define the \emph{local Hodge polygon} $\HP_q(\rho_P^\ext)$ of $\rho$ at $P$ to be the convex polygon with slope set
	\begin{equation*}
		\left\{\frac{1}{d_P}, \dots, \frac{d_P - 1}{d_P}\right\}.
	\end{equation*} 
	By a theorem of Liu-Wei \cite{Liu-Wei}, the Newton polygon $\NP_q(\rho_P^\ext)$ lies above $\HP_q(\rho_P^\ext)$ and both polygons have the same terminal point.
	
	Let $g$ denote the genus of $\overline{X}$. The \emph{global Hodge polygon} $\HP_q(\rho)$ of $\rho$ is obtained by concatenating the local Hodge polygons with $g-1+|S|$ segments of slope $0$ and slope $1$. In other words, $\HP_q(\rho)$ is the the convex polygon with slope set:
	\begin{equation*} 
	\big 
	\{\underbrace{0,\dots,0}_{g-1+|S|}
	\big \}\sqcup \big 
	\{\underbrace{1,\dots,1}_{g-1+|S|}
	\big \}
	\sqcup \bigsqcup_{P \in S} \left\{\frac{1}{d_P}, \dots, \frac{d_P - 1}{d_P}\right\}.
	\end{equation*}
	By a theorem of the first author \cite{Kramer-Miller3}, the Newton polygon $\NP_q(\rho)$ lies above $\HP_q(\rho)$ and both polygons have the same terminal point.
	
	To state our main theorem we introduce the following convention: For each $r > 0$, the \emph{$r$-truncated Newton polygon} $\NP_q^{<r}(\rho)$ is obtained from $\NP_q(\rho)$ by removing all segments of slope $\geq r$. We will use similar notation for all Newton and Hodge polygons attached to $\rho$.
	
	\begin{theorem}
		\label{intro theorem: local touching of polygons givesglobal touching}
		Assume that $X$ is ordinary. Let $r \in [0,1]$. Then $\HP_q^{< r}(\rho)$ and $\NP_q^{< r}(\rho)$ have the same terminal point
		if and only if $\HP_q^{< r}(\rho_P^\ext)$
		and $\NP_q^{< r}(\rho_P^\ext)$ have the same terminal point for each $P \in S$.
	\end{theorem}

	\begin{remark}
		In Theorem \ref{t:main} we prove the same result for a much wider class of characters, which we call \emph{$\bm\delta$-Hodge} characters. This general statement will be used to study equicharacteristic $L$-functions
		and $T$-adic $L$-functions in \cite{Kramer-Miller-Upton2} and \cite{Upton2}.
	\end{remark}
	
	As a corollary to Theorem \ref{intro theorem: local touching of polygons givesglobal touching}, we obtain a complete characterization of when the Newton polygon $\NP_q(\rho)$ coincides with $\HP_q(\rho)$. In this paper, we prove the following theorem for $n = 1$. The general case will follow from the $n=1$ case and from our work on $\Z_p$-towers in the sequel \cite{Kramer-Miller-Upton2}:

		\begin{theorem}\label{intro theorem: HP is NP with congruence conditions}
		$\NP_q(\rho) = \HP_q(\rho)$ if and only if
		\begin{enumerate}
			\item	$\overline{X}$ is ordinary.
			\item	$\delta_P = d_P/p^{n-1} \in \Z$ for all $P \in S$.\label{cond:2}
			\item	$p \equiv 1 \pmod{\delta_P}$ for all $P \in S$.\label{cond:3}
		\end{enumerate}
	\end{theorem}
	\begin{remark} \label{r: only if remark}
		The ``only if'' part of Theorem \ref{intro theorem: HP is NP with congruence conditions} can be deduced from
		Theorem \ref{intro theorem: local touching of polygons givesglobal touching} for all $n$. 
		Indeed, assume $\NP_q(\rho)=\HP_q(\rho)$. Then by the Deuring-Shafarevich formula we know that $\overline{X}$ is ordinary.
		From Theorem \ref{intro theorem: local touching of polygons givesglobal touching} we know that
		$\HP_q(\rho_P^\ext)=\NP_q(\rho_P^\ext)$ for each $P \in S$. In particular, we know that 
		$\NP_q(\rho_P^\ext)$ contains the vertex $(1,\frac{1}{d_P})$. However, the vertices of $\NP_q(\rho_P^\ext)$
		are contained in $\frac{1}{p^{n-1}(p-1)}\Z$. This implies $d_P|p^{n-1}(p-1)$.
		To prove \ref{cond:2} and \ref{cond:3} it suffices to show $p^{n-1}|d_P$.
		In particular, it is enough to prove $d_P \geq p^{n-1}$. 
		Let $F_{P,n}/F_P$ be the finite extension corresponding to $\ker(\rho_P)$, so that
		$\Gal(F_{P,n}/F_P)=\Z/p^n\Z$. We then let $F_{P,i}\subset F_{P,n}$ be the field corresponding to
		$p^{n-i}\Z/p^n\Z$. Let $d_{P,i}$ be the largest upper numbering ramification break
		of $\Gal(F_{P,i}/F_P)$ (so that $d_{P,n}$ is $d_P$). By class field theory we know that
		$d_{P,i+1}\geq pd_{P,i}$ (see e.g. \cite[Lemma 5.2]{Kramer-Miller2} for more details), which
		implies $d_{P,n}\geq p^{n-1}$.

	\end{remark}

	\begin{remark}
		We may apply Theorem \ref{intro theorem: HP is NP with congruence conditions} to determine the Newton polygon of certain covers of an ordinary curve $\overline{X}$. Namely, let $X' \to X$ be a finite Galois $\Z/p^n\Z$-cover, which is totally ramified at each $P\in S$. We will assume that for each $P$ the largest ramification break (in upper numbering) of $X'/X$ at $P$ is of the form $\delta_P p^{n-1}$, where $\delta_P \in \Z$ and $p\equiv 1 \mod \delta_P$. The zeta function of $\overline{X}'$ has a product decomposition
		\begin{align} \label{intro eq: decomposing Zeta functions}
			Z(\overline{X}',s) = Z(\overline{X},s) \prod_{\rho} L(\rho,s),
		\end{align}
		where the product is over all nontrivial characters $\rho$ of $\Gal(\overline{X}/X)$. If $\rho$ has order $p^j$, then the Swan conductor of $\rho$ at $P$ is exactly $\delta_Pp^{j-1}$. In particular, we may apply Theorem \ref{intro theorem: HP is NP with congruence conditions} to each $L$-function in \eqref{intro eq: decomposing Zeta functions} to completely determine the Newton polygon of $\overline{X}'$.
		
		In the case of an Artin-Schreier cover, we are able to recover a result of Booher and Pries (see \cite[Corollary 4.3]{Booher-Pries}). In fact, we obtain a much stronger result: In \cite{Booher-Pries}, Booher and Pries show
		that there \emph{exists} an Artin-Schreier cover $X' \to X$ such that the $q$-adic Newton polygon of $\overline{X}'$ has slope set
		\begin{equation*} 
			\big 
			\{\underbrace{0,\dots,0}_{\mathclap{pg+(p-1)(|S|-1)}}
			\big \}~\sqcup~ \big 
			\{\underbrace{1,\dots,1}_{\mathclap{pg+(p-1)(|S|-1)}}	
			\big \}
			\sqcup \bigsqcup_{P \in S} \left\{\frac{1}{d_P}, \dots, \frac{d_P - 1}{d_P}\right\}^{\mathrlap{\times (p-1)}},
		\end{equation*}
		where the $\times n$ superscript means each slope appears with multiplicity $n$. By Grothendeick's specialization theorem, this 
		determines the Newton polygon for a generic $\Z/p\Z$-cover of $X$ with these ramification invariants.
		However, by the above paragraph we see that this generic Newton polygon occurs for
		\emph{every} Artin-Schreier cover $X' \to X$ with the same ramification breaks.
	\end{remark}

\subsection{Outline of proof}

The starting point of our proof is the method developed by the first author in \cite{KramerMiller} and \cite{Kramer-Miller3}.
However, the methods in this earlier work only produce lower bounds on $\NP_q(\rho)$, and are woefully inadequate 
for pinpointing exact vertices. The necessary new ingredient is a perturbation theory for $p$-adic operators, which allows us to relate the Newton polygons $\NP_q(\rho)$ and $\NP_q(\rho_P^\ext)$.

\paragraph{Characters and Liftings}

Our general approach is to study $L(\rho,s)$ via the Monsky trace formula, which expresses $L(\rho,s)$ as a ratio of Fredholm determinants of operators. This requires
us to lift the Frobenius on $U$ to characteristic $0$. We use the lifting construction developed by the first author in \cite{KramerMiller}. Choose a tame Belyi map $\eta:\overline{X} \to \PP_{\F_q}^1$ for which $\eta(P) = 0$ for all $P \in S$. Let $U \subseteq X$ denote the \'etale locus of $\eta$. Let $A$ be a smooth lifting of the coordinate ring of $U$ to characteristic $0$, and write $A^\dagger$ for the $p$-adic weak completion of $A$ (see \S \ref{ss:growth}). Using $\eta$ we construct a lifting of Frobenius $\sigma:A^\dagger \to A^\dagger$ whose local expression at each $P \in S$ is particularly simple, see \S\ref{ss:semilocal}.

Let $\rho_U$ denote the restriction of $\rho$ to $\pi_1(U)$. Let $\pi \in \Z_p[\rho]$ be a uniformizer, and consider the base change
\begin{equation*}
	A_\pi^\dagger	=	\Z_p[\rho]	\otimes_{\Z_p}	A^\dagger.
\end{equation*}
In this setting, the Monsky trace formula states that there is a $\Z_q[\rho]$-linear endomorpshism $\Theta_q:A_\pi^\dagger \to A_\pi^\dagger$ for which
\begin{align*}
	L(\rho_U,s) &= \frac{\det(1-s\Theta_q | A_\pi^\dagger[\tfrac{1}{\pi}])}{\det(1-qs \Theta_q | A_\pi^\dagger[\tfrac{1}{\pi}])}.
\end{align*}

\paragraph{The Space $A_\pi^\dagger$}

It follows from the trace formula that the truncated Newton polygon $\NP_q^{<1}(\rho_U)$ agrees with that of the characteristic series $\det(1-s\Theta_q | A_\pi^\dagger[\tfrac{1}{\pi}])$. Thus, our goal in this paper is to estimate the Newton polygon of this series. In the classical situation $X = \mathbb{G}_m$, the ring $A^\dagger$ may be identified with the $p$-adic weak completion of the ring of Laurent polynomials $\Z_q[t^\pm]$. The $\Z_q$-module structure of $A^\dagger$ is quite explicit, as the elements $t^k$ for $k \in \Z$ form a natural topological basis. Choosing a suitable basis for the lifting $A^\dagger$ is a key point of the papers \cite{KramerMiller} and \cite{Kramer-Miller3}, where a Riemann-Roch argument is used to construct a basis with prescribed poles at each branch point $P$ of $\eta$.

For each such $P$, choose a local parameter $t_P$ at $P$ and consider the ring $\calA_P = \Z_q(\!(t_P)\!)$. This ring is a lifting of the local field $F_P$ to characteristic $0$, and there is a unique map $A^\dagger \to \calA_P^\dagger$ given by ``expanding functions in the parameter $t_P$''. The Riemann-Roch theorem implies that there is an exact sequence of $\Z_q$-modules
\begin{equation*}
	0	\to	L	\to	A^\dagger	\to	\bigoplus_P \calA_P^{\dagger,\tr}	\to	0,
\end{equation*}
where $L$ is a free $\Z_q$-module of finite rank, and $\calA_P^{\dagger,\tr}$ is a space of truncated Laurent series in $\calA_{\pi,P}$ with prescribed poles. The space $\calA_P^{\dagger,\tr}$ admits a topological basis of the form $B_P = \{t_P^{-k}: k > \mu(P)\}$, where $\mu(P)$ is a parameter depending on the branching of the map $\eta$ at $P$. Is is then straightforward to ``lift'' these bases to produce a global basis $B$ for $A^\dagger$.
%

\paragraph{Perturbing $\pi$-adic Operators}

Let $P$ be a branch point of $\eta$, and let $\calA_{\pi,P} = \Z_q \otimes_{\Z_p} \calA_P$. The operator $\Theta_q$ extends in a natural way to the $\pi$-adic weak completion $\calA_{\pi,P}^\dagger$. If $P \in S$, then the truncated space $\calA_{\pi,P}^{\dagger,\tr}$ is invariant with respect to the action of $\Theta_q$\footnote{Strictly speaking, this is only true after replacing $\Theta_q$ with a suitable ``twist.'' This is a technical point which we ignore in this introduction.}, and we may therefore speak of the local Fredholm determinant
\begin{equation*}
	\det(I-s \Theta_q|\calA_{\pi,P}^{\dagger,\tr}[\tfrac{1}{\pi}])  \in 1+s\Z_q[\rho]\llbracket s\rrbracket.
\end{equation*}
In \S \ref{ss:localtoglobal}, we relate the Newton polygon of this series to that of the local extension $\rho_P^\ext$. In particular, we are able to give a lower bound for this Newton polygon in terms of the local Hodge polygon $\HP_q(\rho_P^\ext)$.

The key insight of this paper is that the Newton polygon of $\Theta_q$ acting on $A_\pi^\dagger$ should be closely related to the local Newton polygons at the ramified points $P \in S$. Let $\Psi = \Theta_q|B$ denote the matrix of $\Theta_q$ with respect to our chosen global basis. In \S \ref{ss:global touching} we construct a new matrix $\Psi'$ in terms of the local matrices $\Theta_q|B_P$ at each $P \in S$. To quantify the close relationship between $\Psi$ and $\Psi'$, we introduce in \S \ref{ss:perturbation} the notion of a \emph{$\pi$-adic $r$-perturbation}. The $r$-perturbation condition implies the following relation between the Newton and Hodge polygons of both matrices:
\begin{align*}
	\begin{array}{l}
	\text{Touching between the $\pi$-adic Hodge} \\
	\text{and Newton polygons of $\Psi$ for}\\
	\text{slopes smaller than $r$}
	\end{array} &\iff \begin{array}{l}
	\text{Touching between the $\pi$-adic Hodge} \\
	\text{and Newton polygons of $\Psi'$ for}\\
	\text{slopes smaller than $r$}
	\end{array} 
\end{align*}
In \S \ref{ss:global touching} we show using the estimates of Section \ref{s:estimates} that $\Psi'$ is an $r$-perturbation of $\Psi$ for appropriate $r$. This allows us to deduce Theorem \ref{intro theorem: local touching of polygons givesglobal touching}.

		\subsection{Applications to \texorpdfstring{$\Z_p$}{Zp}-towers of curves}
		Our primary motivation for this article is to study the $p$-adic variation of $L$-functions
		along $\Z_p$-towers of curves. A \emph{$\Z_p$-tower} $X_\infty/X$ is a sequence of finite Galois coverings
		\begin{equation*}
		\cdots 	\to X_2	\to X_1 \to X_0 =  X,
		\end{equation*}
		together with compatible identifications $\Gal(X_n/X) \cong \Z/p^n \Z$ for all $n$. Equivalently, a $\Z_p$-tower over $X$ corresponds to a continuous and surjective map $\rho:\pi_1(X) \to \Z_p$. We will typically regard $X_\infty/X$ as a family of characters of $\pi_1(X)$ as follows: For every continuous $p$-adic character $\chi \in \Hom(\Z_p,\C_p^\times)$, we obtain a composite character
		\begin{equation*}
		\rho_\chi = \chi \circ \rho:\pi_1(X)	\to	\C_p^\times.
		\end{equation*}
		Daqing Wan has devised a program (see \cite{Wan5}) to understand how the Newton polygon $\NP_q(\rho_\chi)$ varies as a function of $\chi$. The general philosophy is that for well behaved towers we expect
		the Newton polygons to exhibit a certain degree of regularity as $\chi$ varies. Recently, there has been much progress
		surrounding the case $X=\mathbb{A}^1$ (see e.g. \cite{Davis} for one of the first articles on the topic or \cite{KZ} for results on a general class of towers). However, essentially nothing is known for higher genus curves
		or when multiple points along the tower are ramified. The local-to-global results of this article allow us to
		approach these questions when the base curve is ordinary. This is the content of the sequel articles \cite{Kramer-Miller-Upton2} and \cite{Upton2}.
		
		\subsubsection{\texorpdfstring{$\Z_p$}{Zp}-Towers with Strictly Stable Monodromy}\label{ss:stable}
		
		Let $\bm\delta = (\delta_P)_{P \in S}$ be a tuple of positive rational numbers in $\Z[\tfrac{1}{p}]$. We say that the tower $X_\infty/X$ has \emph{$\bm\delta$-stable monodromy} if for each $P \in S$, the highest ramification break of $X_n/X$ at $P$ is of the form $p^{n-1} \delta_P$ for all $n \gg 0$. 
		
		\begin{theorem}[\cite{Kramer-Miller-Upton2}]\label{t:zptowers}
			Suppose that $\overline{X}$ is ordinary. Let $X_\infty/X$ be a $\Z_p$-tower of curves with $\bm\delta$-stable monodromy. Then for a finite character $\chi:\Z_p \to \C_p^\times$ of order $p^n$:
			\begin{enumerate}
				\item	\label{i:uniformity}(Slope Uniformity) The $q$-adic Newton slopes of $\rho_\chi$ are equidistributed in the interval $[0,1]$ as $n \to \infty$.
				\item	\label{i:stability}(Slope Stability) Suppose that $\delta_P \in \Z$ and that $\delta_P \equiv 1 \pmod{p}$ for all $P \in S$. Then $\NP_q(\rho_\chi)$ has slope set
				\begin{equation*}
				\{ \underbrace{0,...,0}_{g-1+|S|}	\}	\sqcup	\{ \underbrace{1,...,1}_{g-1+|S|}	\}	\sqcup	\bigsqcup_{P \in S}	\left\{	\frac{1}{\delta_P p^{n-1}}, \dots \frac{\delta_P p^{n-1} -1}{\delta_P p^{n-1}}	\right\}.
				\end{equation*}
				In particular, the invariants $\delta_1,...,\delta_P$ completely determine the Newton polygon of each curve $X_n$. 
			\end{enumerate}
		\end{theorem}
		This theorem generalizes work of Kosters-Zhu (see \cite{KZ}) for $\bm\delta$-stable towers over $\mathbb{A}_{\F_q}^1$. 
		To the best of our knowledge, there were no prior examples of slope uniform or slope stable towers beyond the $\mathbb{A}_{\F_q}^1$ case. Using the methods of \cite{Kramer-Miller3} we obtain similar theorems for $\bm\delta$-stable towers twisted by a tame finite character $\psi:\pi_1(X) \to \Z_p^\times$.
		
		\subsubsection{Overconvergent \texorpdfstring{$\Z_p$}{Zp}-Towers}\label{ss:oc}
		
			Any tower $X_\infty/X$ is determined by an Artin-Schreier-Witt equation $F(x)-x=f$, where $f$ is an element of the $p$-typical Witt vectors of the coordinate ring of $X$. We say that $X_\infty/X$ is \emph{overconvergent} if $f$ may be taken to be an overconvergent Witt vector in the sense of \cite{Davis-Langer-Zink}. This is a natural condition on the growth of the poles of the Witt coordinates of $f$, closely related to the notion of overconvergence found in rigid cohomology. We remark that any overconvergent tower is $\bm\delta$-stable for some $\bm\delta$, but that most $\bm\delta$-stable towers are not overconvergent.
		
		A notable class of examples are the towers over $X = \A_{\F_q}^1$ studied in work of Davis-Wan-Xiao \cite{Davis}. Those authors prove that a special class of towers are slope-uniform and slope-stable. In fact, they prove a beautiful \emph{spectral halo theorem} which describes the $p$-adic variation of the zero locus of $L(\rho_\chi,s)$ as $\chi$ varies through the character space $\Hom(\Z_p,\C_p^\times)$.  In the final part of this series of articles \cite{Upton2}, the second author shows that the spectral halo theorem holds for a general overconvergent tower $X_\infty/\A_{\F_q}^1$. By combining this result with our local-to-global theorems, we obtain the following stability result for overconvergent towers over an ordinary curve:
		
		\begin{theorem}
			Suppose that $X$ is ordinary. Let $X_\infty/X$ be an overconvergent $\Z_p$-tower of curves. There exist a non-negative integer $m_0$ and postivie rational numbers $\alpha_1,...,\alpha_d$ with the following property: For every $\chi$ of order $p^n$ with $n > m_0$, the slope set of $\NP_q(\rho_\chi)$ is
			\begin{equation*}
			\{ \underbrace{0,...,0}_{g-1}	\}	\sqcup	\{ \underbrace{1,...,1}_{g-1+|S|}	\}	\sqcup	\bigsqcup_{P \in S}\bigsqcup_{i = 0}^{p^{n-m_0}-1}	\left\{	\frac{i}{p^{n-m_0-1}},\frac{\alpha_1+i}{p^{n-m_0-1}},...,\frac{\alpha_d+i}{p^{n-m_0-1}}	\right\}.
			\end{equation*}
		\end{theorem}
		
		As a byproduct of our approach, we obtain similar estimates for the Newton polygon $\NP_q^{< 1}(\rho_\chi)$ for \emph{any} $\chi$ (finite or infinite) near the boundary of the character space $\Hom(\Z_p,\C_p^\times)$. This may be regarded as a ``small-slope'' analogue of the spectral halo theorem of Davis-Wan-Xiao. Unfortunately, our perturbation approach does not produce any information regarding the zeroes of $L(\rho_\chi,s)$ outside of the disk $v_p(s) > - v_p(q)$. It seems likely that new techniques are necessary to study zeros in this region.
		
\subsection{Acknowledgments} 
	We would like to thank Daqing Wan for sharing his insight on exponential sums and towers.
	In addition, we acknowledge Kiran Kedlaya for discussing this work with us. 

\section{Growth Conditions}

	Throughout this article, our basic coefficient ring will be a complete discrete valuation ring $R$ with maximal ideal $\frakm$ and residue field $\F_p$ (we do not assume that $R$ has characteristic $0$). We fix a non-zero topologically nilpotent element $\pi \in R$. We will frequently work with algebras of functions over $R$, which are $\pi$-adically overconvergent. In this section we spell out our conventions regarding these ``growth conditions'' and introduce some basic constructions.
	
	\subsection{Weak Completions}\label{ss:growth}
	
		Let $A$ be an $R$-algebra equipped with the $\pi$-adic topology. We will write $A^\infty$ for the $\pi$-adic completion of $A$. Given real numbers $m$, $b$ with $m > 0$, let $A^m(b)$ denote the set of $a \in A^\infty$ with the following property: there exist $a_1,...,a_n \in A$ and polynomials $p_j \in \pi^jR[X_1,...,X_n]$ such that
		\begin{equation}\label{eq:wc}
			a = \sum_j p_j(a_1,...,a_n)
		\end{equation}
		and $\deg(p_j) < mj + b$. Then each $A^m(b)$ is a $\pi$-adically complete submodule of $A^\infty$. The following basic properties are immediate: 
		
		\begin{lemma}\label{l:Amb}
			For each $m > 0$,
			\begin{enumerate}
				\item	$A^m(b_1) \cdot A^m(b_2) \subseteq A^m(b_1+b_2)$.
				\item	$\pi^k A^m(b) \subseteq A^m(b-km)$.
			\end{enumerate}
		\end{lemma}
		
		\begin{definition}
			The \emph{weak completion} $A^\dagger$ of $A$ is the union of the submodules $A^m(b)$.
		\end{definition}
		
		It is clear from Lemma \ref{l:Amb} that $A^\dagger$ is a subring of $A^\infty$ and is independent of the choice of $\pi \in R$. By \cite[Theorem 1.4]{MW}, the natural map $A/\pi A \to A^\dagger/\pi A^\dagger$ is an isomorphism. Let us say that a subset $S \subseteq A^\dagger$ is a set of \emph{weak generators} for $A^\dagger$ if every $a \in A^\dagger$ has the form (\ref{eq:wc}), where the $a_i$ are elements of $S$. For our purposes we will only consider weak completions in two situations:
		\begin{enumerate}
			\item	Let $A$ be an $R$-algebra such that $A^\dagger$ admits a finite set of weak generators. In this case a theorem of Fulton \cite{Fulton} guarantees that $A^\dagger$ is Noetherian.
			\item	Let $A = R_q(\! (t)\! )$ be the ring of formal Laurent series over $R_q$. Then
			\begin{equation*}
				A^\dagger	\cong	(R_q\llbracket t \rrbracket \otimes_R R[X]^\dagger)/(tX-1).
			\end{equation*}
			This is once again a Noetherian ring. In this case we have:
			\begin{equation*}
				A^m(b)	=	\left\{	\sum_{k = -\infty}^\infty a_k t^{-k} : v_\pi(a_k) \geq \frac{k-b}{m} \text{ for }k>0	\right\}.
			\end{equation*}
		\end{enumerate}
		
	\subsection{Weak Base Change}
		
		A typical situation for us will be as follows: Let $A$ be a $\Z_p$-algebra, equipped with the $p$-adic topology. We define the base change $A_\pi = R \otimes_{\Z_p} A$, which we equip with the $\pi$-adic topology. We will always assume that the modules $A^m(b)$ are given by $p$-adic growth conditions, whereas the $A_\pi^m(b)$ are given by $\pi$-adic growth conditions (hence the choice of subscript).
		
		Since many of our constructions will be given initially over $\Z_p$, it is convenient to have a ``comparison'' between the weak completions $A^\dagger$ and $A_\pi^\dagger$. First, we require a basic result from commutative algebra:
		
		\begin{lemma}\label{l:red}
			Consider an exact sequence of $R$-modules
			\begin{equation}\label{eq:ses}
				0	\to	L	\to	M	\xrightarrow{f}	N.
			\end{equation}
			Let $\overline{f}:\overline{M} \to \overline{M}$ denote the reduction of $f$ modulo $\pi$. Then:
			\begin{enumerate}
				\item	\label{i:surj}If $M$ is $\pi$-adically complete, $N$ is $\pi$-adically separated, and $\overline{f}$ is surjective, then $f$ is surjective.
				\item	\label{i:ker}If $N$ is $\pi$-torsion free and $f$ is surjective, then $\overline{L} = \ker(\overline{f})$.
				\item	\label{i:inj}If $M$ and $N$ are $\pi$-torsion free, $M$ is $\pi$-adically separated, and $\overline{f}$ is injective, then $f$ is injective.
			\end{enumerate}			
		\end{lemma}
		\begin{proof}
			Let $n \in N$. Since $\overline{f}$ is surjective, there exists $m_0 \in M$ such that $n = f(m_0) + \pi n_1$, for some $n_1 \in N$. Repeating this procedure, we inductively construct $m_i \in M$ such that $n_i = f(m_i) + \pi n_{i+1}$. Since $M$ is $\pi$-adically complete, the sum
			\begin{equation*}
				m = \sum_{i = 0}^\infty m_i \pi^i
			\end{equation*}
			converges in $M$. The difference $f(m)-n$ vanishes because $N$ is $\pi$-adically separated, proving (\ref{i:surj}). Since $R$ is a discrete valuation ring, if $N$ is $\pi$-torsion free then $N$ is flat over $R$. Claim (\ref{i:ker}) then follows by tensoring (\ref{eq:ses}) with $R/\pi R$.
			
			Suppose then that $M$ and $N$ are $\pi$-torsion free, $M$ is $\pi$-adically separated, and that $\overline{f}$ is injective. Since $M$ and $N$ are flat over $R$, by tensoring with the exact sequence
			\begin{equation*}
				0	\to	\pi^{j+1} R	\to	\pi^j R	\to	\pi^j R/\pi^{j+1} R	\to	0,
			\end{equation*}
			we have an identification $\pi^j M/\pi^{j+1} M \cong \pi^j R/\pi^{j+1} R	\otimes_{R/\pi R} \overline{M}$ and similarly for $N$. Since $R$ is a discrete valuation ring, $\pi^j R/\pi^{j+1} R$ is flat over $R/\pi R$ for all $j \geq 0$. By tensoring $\overline{f}$ with $\pi^j R/ \pi^{j+1} R$, and using the above identification, we see that $f$ induces an injective map
			\begin{equation*}
				\pi^j M/\pi^{j+1} M	\to	\pi^j N/\pi^{j+1} N.
			\end{equation*}
			It follows that the kernel of $f$ lies in $\bigcap_j \pi^j M = 0$. Since $M$ is $\pi$-adically separated, this proves (\ref{i:inj}).
		\end{proof}
		
		By base change, the natural map $A \to A^\dagger$ induces a map $A_\pi \to (A^\dagger)_\pi$. Passing to weak completions, we obtain a map
		\begin{equation}\label{eq:wbc}
			A_\pi^\dagger	\to	(A^\dagger)_\pi^\dagger.
		\end{equation}
		Our ``comparison'' between $A^\dagger$ and $A_\pi^\dagger$ is provided by the following:
		
		\begin{lemma}[Weak Base Change]
			Suppose that $A$ is flat over $\Z_p$. Then the map (\ref{eq:wbc}) is an isomorphism.
		\end{lemma}
		\begin{proof}
			The flatness condition guarantees that $A_\pi^\dagger$ and $(A^\dagger)_\pi^\dagger$ are $\pi$-torsion free and $\pi$-adically separated. By Lemma \ref{l:red}, it suffices to prove that (\ref{eq:wbc}) is surjective. Since $(A^\dagger)_\pi$ is a set of weak generators for $(A^\dagger)_\pi^\dagger$, and since the image of (\ref{eq:wbc}) is weakly complete, we only need to show that the image contains $(A^\dagger)_\pi$. Now $A$ constitutes a set of weak generators for $A^\dagger$, so every element of $(A^\dagger)_\pi$ is a finite sum of elements of the form
			\begin{equation*}
				r \otimes \sum_j p_j(a_1,...,a_n),
			\end{equation*}
			where $r \in R$, $a_1,...,a_n \in A$ and $p_j \in p^j \Z_p[X_1,...,X_n]$ has degree $\leq mj+b$. Since $p$ is topologically nilpotent in $R$, we see that this is exactly the image of the convergent sum
			\begin{equation*}
				\sum_j (r \otimes p_j)(a_1,...,a_n) \in A_\pi^\dagger,
			\end{equation*}
			where we regard $r \otimes p_j$ as a polynomial in $R[X_1,...,X_n]$.
		\end{proof}

\section{Fredholm Theory}\label{s:newton}

	Let $K = R[\tfrac{1}{\pi}]$ denote the field of fractions of $R$. If $q$ is a power of $p$, then we define the ring $R_q = \Z_q \otimes_{\Z_p} R$ and its field of fractions $K_q = K \otimes_R R_q$. We will now discuss the Fredholm theory of nuclear operators acting on $K_q$-vector spacse and the corresponding theory of Newton and Hodge polygons of such operators.
		
	\subsection{Spectral Theory of Completely Continuous Operators}
	
		\subsubsection{Normed Vector Spaces}
	
		\begin{definition}
			A \emph{normed vector space} over $K_q$ is a $K_q$-vector space of the form $V = K \otimes_R M$, where $M$ is a $\pi$-adically separated $R_q$-module. We say that $V$ is a \emph{Banach space} over $K_q$ if $M$ is $\pi$-adically complete.
		\end{definition}
		
		Let $V = K \otimes_R M$ be a normed vector space over $K_q$. We topologize $V$ by taking the $R_q$-submodules $\frakm^n M$ to be a fundamental system of neighborhoods of $0$ in $V$. As the terminology suggests, the choice of $M \subseteq V$ uniquely determines a norm $\|\cdot \|_\pi$ on $V$ defined as follows: If $\pi R = \frakm^d R$, then we define the \emph{$\pi$-adic valuation} on $V$ via
		\begin{equation*}
			v_\pi(x)	=	\frac{1}{d}\sup\{n \in \Z:x \in \frakm^n M\}.
		\end{equation*}
		The condition that $M$ is $\pi$-adically separated guarantees that $\|x \|_\pi = p^{-v_\pi(x)}$ defines a norm on $V$. Since we can recover $M$ as the $R_q$-submodule of $x \in V$ with $\|x\|_\pi \leq 1$, the choice of $M$ is equivlent to the norm $\|\cdot \|_\pi$ or the valuation $v_\pi$.
		
		\begin{example}
			Let $I$ be a set. We define $b(I)$ to be the $K_q$-vector space of $I$-families $x = (x_i)_{i \in I}$ of elements of $K_q$ which are $\pi$-adically bounded, in the sense that
				\begin{equation*}
					\|x\|_\pi = \sup_{i \in I} \|x_i\| < \infty.
				\end{equation*}
				Then $b(I)$ is a Banach space over $K_q$. For each $i \in I$, let $e_i \in b(I)$ denote the element whose $i$th component is 1, and whose other components are all $0$. Then we may represent an element $x \in b(I)$ as a formal sum
				\begin{equation}\label{eq:formalsum}
					x	=	\sum_{i \in I} x_i e_i.
			\end{equation}
			Let $c(I)$ be the subspace of $b(I)$ consisting of those $x$ for which $x_i \to 0$ in the finite-complement topology on $I$. Then $c(I)$ is a Banach space over $K_q$, and for each $x \in c(I)$ the formal sum (\ref{eq:formalsum}) converges to $x$ in the norm topology on $c(I)$.
		\end{example}
		
		Let $V = K \otimes_R M$ and $W = K \otimes_R N$ be normed vector spaces over $K_q$. Let $M \hat{\otimes}_{R_q} N$ denote the $\pi$-adic completion of the $R_q$-module $M \otimes_{R_q} N$. We define the \emph{completed tensor product} of $V$ and $W$ to be the $K_q$-Banach space $V \hat{\otimes}_{K_q} W = K \otimes_R (M \hat{\otimes}_{R_q} N)$. Similarly, for each $n \geq 0$ we let $\wedge^n M$ denote the $\pi$-adic completion of the $n$th exterior power of $M$ (over $R_q$). We define the $n$th \emph{completed exterior power} of $V$ to be the $K_q$-Banach space $\wedge^n V = K \otimes_R \wedge^n M$.
		
		We write $\Hom(V,W)$ for the $K_q$-vector space of \emph{continuous} $K_q$-linear maps $V \to W$, equipped with the usual norm
		\begin{equation*}
			\|\psi\|_\pi	=	\sup_{x \in M} \| \psi(x) \|_\pi < \infty.
		\end{equation*}
		In particular, we will write $V^* = \Hom(V,K)$ for the \emph{continuous} dual space of $V$.
			
		\subsubsection{Completely Continuous Maps}
		
		\begin{definition}
			Let $V$ and $W$ be Banach spaces over $K_q$. We say that a continuous $K_q$-linear map $\psi:W \to V$ is \emph{completely continuous} if $\psi$ is a limit of operators of finite rank. Equivalently, $\psi$ is completely continuous if and only if $\psi$ lies in the image of the natural map
			\begin{equation*}
				W^* \hat{\otimes}_{K_q} V \to \Hom(W,V).
			\end{equation*}
		\end{definition}
		
		In particular, we identify the space of completely continuous operators $V \to V$ with the Banach space $V^* \hat{\otimes}_{K_q} V$. The \emph{trace map} on this space is defined in the usual way:
		\begin{align*}
			\Tr:V^* \hat{\otimes}_{K_q} V	&\to	K_q	\\
			f \otimes v	&\mapsto	f(v).
		\end{align*}
		If $\psi:V \to V$ is completely continuous, then for each $n\geq 0$ the $n$th exterior power $\wedge^n \psi$ is a completely continuous operator on $\wedge^n V$.
		
		\begin{definition}
			Let $\psi:V \to V$ be a completely continuous operator. For each $n \geq 0$ we define the $n$th \emph{Fredholm coefficient} to be
			\begin{equation*}
				c_n(\psi)	=	(-1)^n \Tr(\wedge^n \psi).
			\end{equation*}
			The \emph{Fredholm determinant} of $\psi$ is the power series
			\begin{equation*}
				\det(I - s \psi)	=	\sum_{n = 0}^\infty c_n(\psi) s^n	\in	1+sK_q\llbracket s\rrbracket.
			\end{equation*}
		\end{definition}
		
		\begin{lemma}[{\cite[Proposition 15]{Serre}}]\label{l:ccdual}
			Let $\psi:V \to V$ be a completely continuous operator. Then the dual map $\psi^*:V^* \to V^*$ is completely continuous, and there is an equality of Fredholm series
			\begin{equation*}
				\det(I-s\psi) = \det(I-s\psi^*).
			\end{equation*}
		\end{lemma}
		
		\subsubsection{Orthonormal and Integral Bases}
		
			\begin{definition}
				Let $V$ be a Banach space over $K_q$. Let $I$ be a set and let $B = \{e_i:i \in I\}$ be a family of elements of $V$ indexed by $I$. We say that $B$ is an \emph{orthonormal basis} for $V$ if, for every $x \in V$ there exists a unique $(x_i) \in c(I)$ such that  $\|x\|_\pi = \sup_{i \in I} |x_i|_\pi$ and
				\begin{equation*}
					x = \sum_{i \in I} x_i e_i.
				\end{equation*}
			\end{definition}
			
			Let $V = K_q \otimes_{R_q} M$ be a Banach space over $K_q$ and let $B = \{e_i:i \in I\}$ be subset of $V$. By \cite[Proposition 1]{Serre}, $B$ is an orthonormal basis for $V$ if and only if $B \subseteq M$ and $B$ reduces to an $\F_q$-basis for $\F_q \otimes_{R_q} M$. It follows that every Banach space over $K_q$ admits an orthonormal basis. An orthonormal basis for $V$ indexed by $I$ is equivalent to an isometric isomorphism of Banach spaces $\iota:V \xrightarrow{\sim} c(I)$, sending each $x \in V$ to the corresponding $I$-family $(x_i) \in c(I)$.
		
		\begin{definition}
			Let $V$ be a normed vector space over $K_q$. A \emph{formal basis} for $V$ is a pair $(B,\iota)$, where $B = \{e_i:i \in I\}$ is a subset of $V$, and $\iota:V \to b(I)$ is an isometric embedding sending $e_i \mapsto e_i$ for all $i \in I$.
		\end{definition}
		
		Typically we will abuse notation and refer to a formal basis $(B,\iota)$ simply by $B$, leaving the embedding into $b(I)$ implicit.
		
		\begin{example}\label{ex:dual}
			Let $V$ be a Banach space and let $B = \{e_i:i \in I\}$ be an orthonormal basis for $V$. By the preceding discussion, $B$ is a formal basis for $V$ in a natural way. For each $i \in I$, consider the continuous linear functional $e_i^* \in V^*$ defined by $e_i^*(x) = x_i$. Let $B^* = \{e_i^*:i \in I\} \subset V^*$. If $f \in V^*$ is any continuous linear functional, then we define
			\begin{equation*}
				\iota(f)	=	(f(e_i)) \in b(I).
			\end{equation*}
			Then $\iota:V^* \cong b(I)$ is an isometric isomorphism of Banach spaces over $K_q$, and the pair $(B^*,\iota)$ is a formal basis for $V^*$.
		\end{example}
		
		Let $I$ be a set. Let $\Psi = (\psi_{i,j})$ be an $I \times I$ matrix with entries in $K_q$. We say that $\Psi$ is of \emph{trace class} if $\psi_{i,i} \to 0$ in the finite-complement topology on $I$. In this case, we define the \emph{trace} of $\Psi$ to be
		\begin{equation*}
			\Tr(\Psi)	=	\sum_{i \in I} \psi_{i,i} \in K_q.
		\end{equation*}
		Choose any ordering on $I$. For each $n \geq 0$, the exterior power $\wedge^n \Psi$ is a matrix with entries in $K_q$ indexed by the set
		\begin{equation*}
			\wedge^n I	=	\{(i_1,...,i_n): i_1 < \cdots < i_n\}.
		\end{equation*}
		For each subset $J \subseteq I$, let $\Psi_J$ denote the $J \times J$ submatrix of $\Psi$ corresponding to $J$.
		
		\begin{definition}
			Suppose that $\wedge^n \Psi$ is of trace class for all $n \geq 0$. We define the $n$th \emph{Fredholm coefficient} of $\Psi$ to be
			\begin{equation}\label{eq:coeff}
				c_n(\Psi) = (-1)^n \Tr(\wedge^n \Psi) = (-1)^n\sum_{|J| = n} \det(\Psi_J).
			\end{equation}
			The \emph{Fredholm determinant} of $\Psi$ is the power series
			\begin{equation*}
				\det(I - s \Psi)	=	\sum_{n = 0}^\infty c_n(\Psi) s^n	\in	1+sK_q\llbracket s\rrbracket.
			\end{equation*}
		\end{definition}
		
		\begin{lemma}[{\cite[\S 5]{Serre}}]\label{l:ccmatrix}
			Let $V$ be a Banach space over $K_q$ and let $\psi:V \to V$ be a completely continuous operator. Let $B = \{e_i:i \in I\}$ be an orthonormal basis for $V$, and let $\psi|B$ denote the $I \times I$ matrix of $\psi$ with respect to $B$. Then there is an equality of Fredholm series
			\begin{equation*}
				\det(I - s\psi) = \det(I-s\psi|B).
			\end{equation*}
		\end{lemma}
		
		\begin{definition}
			We say that an $I \times I$ matrix $\Psi$ with entried in $K_q$ is \emph{tight} if:
			\begin{enumerate}
				\item	For each $i \in I$, the $i$th column $\Psi_i$ of $\Psi$ lies in $b(I)$.
				\item	$\Psi_i \to 0$ in the finite-complement topology on $I$.
			\end{enumerate}			
		\end{definition}
		
		The definition ensures that if $\Psi$ is a tight $I \times I$ matrix and $x \in b(I)$, then the sum
		\begin{equation*}
			\psi(x) = \sum_{i \in I} x_i \Psi_i.
		\end{equation*}
		converges in $b(I)$. Thus every tight $I \times I$ matrix defines an associated continuous $K_q$-linear operator $\psi:b(I) \to b(I)$.
		
		\begin{proposition}\label{p:tight}
			Let $\Psi$ be a tight $I \times I$ matrix, and let $\psi:b(I) \to b(I)$ be the corresponding $K_q$-linear operator. Then $\psi$ is completely continuous, and there is an equality of Fredholm series
			\begin{equation*}
				\det(I-s\psi)	=	\det(I-s\Psi).
			\end{equation*}
		\end{proposition}
		\begin{proof}
			Let $\Psi'$ denote the transpose of $\Psi$. From (\ref{eq:coeff}) we see that $\det(I-s\Psi) = \det(I-s\Psi')$. Since the rows of $\Psi'$ tend to $0$ in the finite-complement topology on $I$, $\Psi'$ defines a \emph{completely continuous} operator $\psi':c(I) \to c(I)$. If we identify $c(I)^* \cong b(I)$ as in Example \ref{ex:dual}, then $\psi = (\psi')^*$. By Lemma \ref{l:ccdual}, we have
			\begin{equation*}
				\det(I-s\psi) = \det(I-s\psi') = \det(I-s\Psi') = \det(I-s\Psi).
			\end{equation*}
		\end{proof}
		
		In light of the proposition, it will generally be convenient to identify the tight $I \times I$ matrix $\Psi$ with the corresponding operator $\psi:b(I) \to b(I)$.
		
		\begin{proposition}\label{p:basis}
			Let $V' \subseteq V$ be a containment of normed vector spaces over $K_q$. Let $\psi:V \to V$ be a $K_q$-linear operator such that $V'$ is $\psi$-invariant. Suppose that there exists a formal basis $B = \{e_i:i \in I\}$ and elements $c_i \in K_q$ such that $B' = \{c_ie_i:i \in I\}$ is a formal basis for $V'$. Then there is an equality of Fredholm series
			\begin{equation*}
				\det(I-s\psi|B) = \det(I-s\psi|B'),
			\end{equation*}
			when either series exists.
		\end{proposition}
		\begin{proof}
			For each finite set $J \subseteq I$ the finite matrices $(\psi|B)_J$ and $(\psi'|B)_J$ are similar, and the result follows from the explicit formula (\ref{eq:coeff}).
		\end{proof}
		
		
	\subsection{Nuclear Operators}
	
		Let $V$ be a $K_q$-vector space and let $\psi:V \to V$ be a $K_q$-linear operator. For every polynomial $g \in 1+sK_q[s]$, we define the subspace
		\begin{equation*}
			V_g	=	\bigcup_{n = 1}^\infty	\ker(g(\psi)^n).
		\end{equation*}
		Let $\overline{K}_q$ be an algebraic closure of $K_q$. For any $\lambda \in \overline{K}_q^\times$, let $g_\lambda = 1+sK_q[s]$ be the irreducible polynomial of $\lambda$ over $K_q$. We say that $\lambda$ is a \emph{non-zero eigenvalue} if $V_{g_\lambda} \neq 0$.
		
		\begin{definition}
			A $K_q$-linear operator $\psi:V \to V$ is \emph{nuclear} if:
			\begin{enumerate}
				\item	For every $g \in 1+sK_q[s]$ there is a $\psi$-equivariant decomposition $V = F_g \oplus V_g$, where $g(\psi):F_g \to F_g$ is bijective and $V_g$ is finite-dimensional. 
				\item	The non-zero eigenvalues of $\psi$ tend to $\infty$ in the finite-complement topology.
			\end{enumerate}			
		\end{definition}
		
		Let $\psi:V \to V$ be a nuclear operator. For each real number $r$ define the subspace
		\begin{equation*}
			V^{<r} = \sum_{v_\pi(\lambda) > -r} V_{g_\lambda},
		\end{equation*}
		Since $\psi$ is nuclear, each $V^{< r}$ is a finite-dimensional subspace of $V$. The \emph{characteristic series} of $\psi$ is defined to be the power series
		\begin{equation*}
			C(\psi|V,s) = \lim_{r \to \infty} \det(I-s\psi|V^{< r}) \in 1+K_q[[s]].
		\end{equation*}
		Over $\overline{K}_q$, this series factors as a product $C(\psi|V,s) = \prod_\lambda (1-s/\lambda)$, where $\lambda$ runs through the non-zero eigenvalues of $\psi$ counted with multiplicity. It follows that $C(\psi|V,s)$ is a $\pi$-adic entire function on $K_q$.
		
		\begin{example}\label{ex:nuclear}
			The following basic examples of nuclear operators, which we will refer to frequently, are due to Monsky \cite{Monsky}:
			\begin{enumerate}
				\item	\label{i:cc}Let $V$ be a Banach space over $K_q$. Let $\psi:V \to V$ be a completely continuous operator. Then by \cite[Theorem 1.3]{Monsky}, $\psi$ is nuclear and
					\begin{equation*}
						C(\psi|V,s)	=	\det(I-s\psi).
					\end{equation*}
				\item	\label{i:quotient}Let $V$ be a $K_q$-vector space, and let $\psi:V \to V$ be a nuclear operator. Let $W \subseteq V$ be a $\psi$-invariant subspace such that $\psi$ restricts to a nuclear operator on $W$. Then by \cite[Theorem 1.4]{Monsky} the action of $\psi$ on $V/W$ is nuclear and 
					\begin{equation*}
						C(\psi|V,s)	=	C(\psi|W,s)\cdot C(\psi|V/W,s).
					\end{equation*}
				\item	\label{i:union}Let $V$ be a $K_q$-vector space. Let $I$ be a linearly ordered set, and let $\{V_i\}_{i \in I}$ be a family of $K_q$-subspaces of $V$ such that $V_j \subseteq V_j$ whenever $i \leq j$, with $V = \bigcup_{i \in I} V_i$. Let $\psi:V \to V$ be a $K_q$-linear operator. Suppose that for all $i \gg 0$, $\psi$ restricts to a nuclear operator on $V_i$, and that the characteristic series $C(\psi|V_i,s)$ is independent of $i$. Then by \cite[Theorem 1.6]{Monsky}, $\psi$ is a nuclear operator on $V$ and for such $i$ we have
					\begin{equation*}
						C(\psi|V,s)	=	C(\psi|V_i,s).
					\end{equation*}
			\end{enumerate}			
		\end{example}
		
%
%
		
	\subsection{Newton and Hodge Polygons}
	
		\begin{definition}
			Let $S$ be a countable multiset of real numbers. We say that $S$ is a \emph{slope set} if for all $r \in R$, the multiset $S^{< r} = \{s \in S : s < r\}$ is finite.
		\end{definition}
		
		Let $S$ be a slope set. If $S$ is finite of cardinality $n$, then there is a unique convex function $f:[0,n] \to \mathbb{R}\sqcup \{\infty\}$ such that:
		\begin{enumerate}
			\item	$f(0) = 0$
			\item	For each $0 \leq i < n$, the restriction of $f$ to $[i,i+1]$ is a linear function.
			\item	$S$ is equal to the multiset $\{f(i+1)-f(i): 0 \leq i < n\}$.
		\end{enumerate}
		In this case, we define the \emph{Newton polygon} $\NP(S)$ of $S$ to be the graph of $f$ in $\R^2$. If $S$ is not finite, then we define the \emph{Newton polygon} of $S$ to be the union
		\begin{equation*}
			\NP(S)	=	\bigcup_r \NP(S^{< r}).
		\end{equation*}
		
		If $S'$ is another slope set, we define the \emph{concatenation} of $S$ and $S'$ to be the slope set $S \sqcup S'$. In particular, we will write $S^{\times n}$ for the $n$-fold concatenation of $S$ with itself. The basic operations on slope sets carry over to their Newton polygons as well. We define a partial order on slope sets (or their Newton polygons) by writing $S_1 \succeq S_2$ whenever the Newton polygon of $S_1$ lies on or above the Newton polygon of $S_2$.
		
		\begin{example}
			\begin{enumerate}
				\item	Let $P(s) = 1+c_1 s + \cdots + c_d s^d$ be a polynomial with coefficients in $K_q$. We define the \emph{$\pi$-adic Newton polygon} $\NP_\pi P(s)$ of $P(s)$ to be the lower convex hull in $\R^2$ of the points $(n,v_\pi(c_n))$ for $0 \leq n \leq d$. Then $\NP_\pi P(s)$ is the Newton polygon of a unique slope set which we call the $\pi$-adic \emph{Newton slopes} of $P(s)$. If we factor
				\begin{equation*}
					P(s)	=	\prod_{i = 1}^d (1-\alpha_i s)
				\end{equation*}
				over $\overline{K}_q$, then the $\pi$-adic Newton slopes of $P(s)$ are precisely the $v_\pi(\alpha_i)$ counted with multiplicity.
				\item	Let $C(s) \in 1+s K_q\llbracket s\rrbracket$ be a power series converging in the disk $v_\pi(s) > -r$. By the Weierstrass preparation theorem there is a factorization
					\begin{equation*}
						C(s)	=	P(s) \cdot C'(s),
					\end{equation*}
					where $P(s) \in 1+sK_q[s]$ has all $\pi$-adic Newton slopes $< r$, and $C'(s)$ converges and is non-zero in the disk $v_\pi(s) > -r$. In this case we define $\NP_\pi^{< r} C(s) = \NP_\pi P(s)$. In particular, if $C(s)$ is entire then we may write
					\begin{equation*}
						\NP_\pi C(s) = \bigcup_r \NP_\pi^{< r} C(s).
					\end{equation*}
				\item	Let $V$ be a $K_q$-vector space and let $\psi:V \to V$ be a nuclear operator. Then we define the \emph{$\pi$-adic Newton polygon} of $\psi$ to be $\NP_\pi (\psi) = \NP_\pi C(\psi,s)$.
			\end{enumerate}
		\end{example}
		
		\begin{definition}
			Let $V$ be a Banach space over $K_q$. Let $\psi:V \to V$ be a completely continuous operator. The $\pi$-adic \emph{Hodge polygon} $\HP_\pi (\psi)$ of $\psi$ is the lower convex hull in $\R^2$ of the points $(n,v_\pi \wedge^n \psi)$.
		\end{definition}
		
		Let $I$ be a set, and let $\Psi$ be a tight $I \times I$ matrix with entries in $K_q$. Then we may write $\HP_\pi(\Psi)$ for the Hodge polygon of the induced operator $b(I) \to b(I)$. For tight matrices, the following ``naive'' Hodge polygon is often useful for making explicit estimates:
		
		\begin{definition}
			The \emph{column Hodge polygon} $\cHP_\pi (\Psi)$ of $\Psi$ is the convex polygon with slope set
			\begin{equation*}
				\{v_\pi \Psi(e_i):i \in I\}.
			\end{equation*}
		\end{definition}
		
		It is straightforward to show that $\HP_\pi(\Psi)$ always lies on or above $\cHP_\pi(\psi)$, and therefore that $\HP_\pi(\Psi)$ is the Newton polygon of a unique slope set. We will frequently make use of the following much stronger estimate:
		
		\begin{lemma}\label{l:hodgeSlopes}
			Let $\Psi$ be a (finite) $n \times n$ matrix with entries in $K_q$. Then every slope of $\HP_\pi(\Psi)$ is greater than or equal to the corresponding slope of $\cHP_\pi(\Psi)$. Consequently, if both polygons pass through a point $(n,m)$ then they agree on the interval $[0,n]$.
		\end{lemma}
		\begin{proof}
			Recall that the ``invariant factor theorem'' states that $\Psi$ can be put in diagonal form by applying finitely many of the following operations: (1) Permuting the rows or columns of $\Psi$, (2) Adding an $R_q$-multiple of one row of $\Psi$ to another, and (3) Given a column with one non-zero entry, adding an $R_q$-multiple of this column to clear other entries in the same row. It is well known that these operations do not affect the Hodge polygon of $\Psi$. On the other hand, each of these operations can only increase the column slopes $v_\pi \Psi(e_i)$. Since $\HP_\pi(\Psi) = \cHP_\pi(\Psi)$ when $\Psi$ is diagonal, this completes the proof.
		\end{proof}
		
		\begin{lemma}\label{l:newtonhodge}
			Let $V$ be a Banach space, and let $\psi:V \to V$ be a completely continuous operator. Then
			\begin{enumerate}
				\item	\label{i:HodgePolygon}$\HP_\pi(\psi)$ is the Newton polygon of a unique slope set.
				\item	\label{i:NewtonOverHodge}$\NP_\pi(\psi) \succeq \HP_\pi(\psi)$.
				\item	\label{i:terminalpt}If $V$ is finite-dimensional and $\det(\psi) \neq 0$, then $\NP_\pi(\psi)$ and $\HP_\pi(\psi)$ have the same terminal point.
			\end{enumerate}
		\end{lemma}
		\begin{proof}
			Let $B = \{e_i:i \in I\}$ be an orthonormal basis for $V$. Let $\Psi = \psi|B$ and let $\Psi'$ denote the transpose of $\Psi$. Then $\Psi'$ is tight, and $\HP_\pi(\psi) = \HP_\pi(\Psi')$. Claim \ref{i:HodgePolygon} then follows from the preceding discussion. Evidently $v_\pi \Tr \wedge^n \psi \geq v_\pi \wedge^n \psi$ for all $n$, proving \ref{i:NewtonOverHodge}. For \ref{i:terminalpt}, simply note that the terminal point of both polygons is $(n, v_\pi \det(\psi))$.
		\end{proof}
		
	\subsection{Iteration}\label{ss:iteration}
	
		Let $F$ denote the canonical Frobenius endomorphism of $\Z_q$. By base change, $F$ induces an $R$-linear endomorphism $F:R_q \to R_q$.
		
		\begin{definition}
			Let $V$ be a vector space over $K_q$. We say that a $K$-linear operator $\psi:V \to V$ is \emph{$F^{-1}$-linear} if
			\begin{equation*}
				\psi(F(r)x) = r\psi(x)
			\end{equation*}
			for all $r \in R$ and $x \in V$. In this case, the iterate $\psi_q = \psi^{v_p(q)}$ is a $K_q$-linear operator.
		\end{definition}
%
		
		\begin{definition}
			Let $V$ be a normed vector space over $K_q$. Let $B_q$ be a formal basis for $V$. An \emph{associated $K$-Basis} for $V$ is a formal basis for $V$ (as a normed vector space over $K$) of the form
			\begin{equation}\label{eq:assocbasis}
				B	=	\bigsqcup_{i = 0}^{v_p(q)-1} F^i(\xi) \otimes B_q,
			\end{equation}
			where $\xi \in \Z_q$ is any element such that the $F^i(\xi)$ generate $\Z_q$ as a $\Z_p$-module.
		\end{definition}
		
		Let $\psi:V \to V$ be an $F^{-1}$-linear operator. Via the partition (\ref{eq:assocbasis}), the matrix $\psi_q|B$ is block diagonal of the form
		\begin{equation*}
			\begin{bmatrix}
				\xi (\psi_q|B_q)	&	&	\\
					&	\ddots	&	\\
					&	&	F^{v_p(q)-1}(\xi)(\psi_q|B_q)
			\end{bmatrix}		.	
		\end{equation*}
		Consider the parameter $\pi_q = \pi^{v_p(q)}$. We have the following relation between the Newton polygons of the operators $\psi$ and $\psi_q$:
		
		\begin{lemma}\label{l:root}
			Let $V$ be a Banach space over $K_q$. Let $\psi:V \to V$ be a completely continuous $F^{-1}$-linear operator. Then $\psi_q$ is completely continuous as a $K_q$-linear operator and
			\begin{equation*}
				\NP_{\pi_q} (\psi_q)^{\times v_p(q)} = \NP_\pi (\psi).
			\end{equation*}
		\end{lemma}
		\begin{proof}
			Note that $\psi_q$ is completely continuous as a $K$-linear operator. Let $B_q$ be an orthonormal $K_q$-basis for $V$, and let $B$ be an associated $K$-basis for $V$. From the above block matrix, we see that the rows of $\psi_q|B_q$ tend to $0$ in the finite-complement topology, thus $\psi_q$ is completely continuous as a $K_q$-linear operator. We see moreover that $\det(I-s\psi_q|B) = \det(I-s\psi_q|B_q)^{v_p(q)}$. It follows that
			\begin{equation*}
				\det(I-s^{v_p(q)}\psi_q|B_q)^{v_p(q)} = \det(I-s^{v_p(q)}\psi_q) = \prod_{\zeta^{v_p(q)}=1} \det(I-s\zeta\psi|B).
			\end{equation*}
			The Newton polygons of the matrices $\zeta \psi|B$ do not depend on $\zeta$. Thus,
			\begin{equation*}
				\NP_{\pi_q} (\psi_q|B_q)^{\times v_p(q)} = \NP_\pi \det(I-s^{v_p(q)}\psi_q|B_q) = \NP_\pi (\psi|B).
			\end{equation*}
			The claim follows from Lemma \ref{l:ccmatrix}.
		\end{proof}
		
\section{Geometric Setup}\label{s:geo}

	Fix a smooth projective curve $\overline{X}$ over $\F_q$ of genus $g$. Let $S$ be a finite set of closed points of $\overline{X}$ and let $X = \overline{X} \backslash S$. In this section we construct the global lifting $(A^\dagger,\sigma)$, which we will use to study abelian $L$-functions over $X$. We will give our constructions over the ring $\Z_p$, with $\pi = p$. Later, we will pass to general $R$ and $\pi$ by weak base change.
		
	\subsection{Mapping to \texorpdfstring{$\mathbb{P}_{\mathbb{F}_q}^1$}{P1}}\label{ss:mapping}
	
		The construction of our global lifting relies on an analogue of Belyi's theorem in characteristic $p$. For this we require:
		
		\begin{theorem}[Fulton \cite{FultonHurwitz}]\label{t:tame}
			After extending the base field, there exists a finite, separable, tamely ramified morphism $\eta_0:\overline{X} \to \PP_{\F_q}^1$.
		\end{theorem}
		\begin{remark}
			In \cite{Sugiyama}, Sugiyama and Yasuda extend Fulton's result to the case $p = 2$. We have omitted this case for other reasons (see Remark \ref{r:2}). By a recent theorem of Kedlaya-Litt-Witaszek, $\eta$ exists even without extending the base field \cite{Kedlaya}.
		\end{remark}
		
		From now on we fix a choice of $\eta_0$ as in Theorem \ref{t:tame}. As in \cite[Theorem 9.3]{Kedlaya}, the existence of $\eta_0$ implies the existence of a tamely ramified map $\overline{X} \to \PP_{\F_q}^1$ with three ramified points. We shall need some control over the branching of this map, so we modify the construction slightly.
		
		By extending the base field, we assume that all branch points of $\eta_0$ and every point in $S$ is $\F_q$-rational. We assume moreover that $q$ is large enough so that there are two $\F_q$-rational points of $\PP_{\F_q}^1$ which are disjoint from the branch points of $\eta_0$ and the image of $S$. Label these points as $0$ and $\infty$, and let $1$ denote any other $\F_q$-rational point of $\PP_{\F_q}^1$. Consider the composition
		\begin{equation*}
			\eta_q:	\PP_{\F_q}^1	\xrightarrow{q-1}	\PP_{\F_q}^1	\to	\PP_{\F_q}^1	\xrightarrow{p-1}	\PP_{\F_q}^1	\to	\PP_{\F_q}^1.
		\end{equation*}
		Here, the first and third maps denote the $(q-1)$- and $(p-1)$-power maps, respectively. The second map is a linear transformation fixing $1$ and $\infty$ and sending $0$ to any other $\F_q$-rational point of $\PP_{\F_q}^1$. The final map is also a linear transformation which fixes $\infty$ and swaps $0$ with $1$. Then $\eta_q$ is ramified over $\{0,1,\infty\}$ and all branch points of $\eta_q$ are $\F_q$-rational. We have:
		
		\begin{proposition}
			The composite $\eta = \eta_q \circ \eta_0:\overline{X} \to \PP_{\F_q}^1$ is a tame Belyi map such that
			\begin{enumerate}
				\item	$\eta(P) = 0$ for each $P \in S$.
				\item	If $\eta(P) = 1$ then the ramification index of $\eta$ at $P$ is $p-1$.
			\end{enumerate}			
		\end{proposition}
		
		For each $Q \in \{0,1,\infty\}$, let $r_Q$ denote the cardinality of the fiber $\eta^{-1}(Q)$. If $\eta(P) = Q$, then we let $e_P$ denote the ramification index of $P$ over $Q$. In particular, if $Q = 1$ then $e_P = p-1$. It follows that that $r_1(p-1) = \deg(\eta)$. The Riemann-Hurwitz formula gives
		\begin{equation}\label{eq:rh}
			2(g-1)+r_0+r_1+r_\infty = \deg(\eta).
		\end{equation}
	
	\subsection{The Global Lifting}\label{ss:global}
		
		Let $V = \PP_{\F_q}^1 \backslash \{0,1,\infty\}$. Let $u_0$ denote a parameter at $0$, and define $u_1 = u_0-1$, $u_\infty = u_0^{-1}$. The coordinate ring of $V$ is a generated by $u_0$, $u_\infty$, and $u_1^{-1}$ over $\F_q$. We consider the lifting
		\begin{equation}\label{eq:B}
			B	=	\Z_q[u_0,u_\infty,u_1^{-1}].
		\end{equation}
		We define $\sigma$ to be the unique Frobenius endomorphism of $B^\dagger$ satisfying $\sigma(u_0) = u_0^p$.
		
		Let $U = \eta^{-1}(V)$. Then $U$ is an affine open subset of $X$, which is finite \'etale over $V$. We define $S_\eta= X-V$. We will lift $U$ to characteristic $0$ rather than all of $X$, because the \'etale property of $\eta$ is essential for constructing our Frobenius endomorphism: First, a theorem of Elkik \cite{Elkik} states that there exists a smooth $\mathbb{Z}_q$-algebra $A$ such that $U = \Spec(A/pA)$. By \cite[Theorem 2.4.4]{Van}, there is a (necessarily unique) lifting of $\eta$ to an \'etale map $B^\dagger \to A^\dagger$.
		
		\begin{definition}
			The \emph{global lifting} is the ring $A^\dagger$ equipped with the unique extension of $\sigma$ along $\eta$.
		\end{definition}
		
		\begin{lemma}\label{l:finiteflat}
			The Frobenius $\sigma:A^\dagger \to A^\dagger$ is injective, and $A^\dagger$ is a finite projective $\sigma(A^\dagger)$-module of rank $p$.
		\end{lemma}
		\begin{proof}
			The first statement is \cite[Theorem 3.2(3)]{MW}. The second follows from that fact that $A^\dagger$ is finite \cite[Theorem 6.2]{MW} and flat \cite[Theorem 2.1(2)]{MW} over $\sigma(A^\dagger)$.
		\end{proof}
		
		Let $\Omega_{A^\dagger}$ denote the $A^\dagger$-module of continuous differentials of $A^\dagger/\Z_q$, and similarly for $B^\dagger$. Then $\Omega_{B^\dagger}$ is a free $B^\dagger$-module of rank $1$ generated by the differential $\tfrac{\D u_0}{u_0}$. Since $\eta:B^\dagger \to A^\dagger$ is \'etale, we have
		\begin{equation*}
			\Omega_{A^\dagger} = A^\dagger \otimes_{B^\dagger} \Omega_{B^\dagger} = A^\dagger \frac{\D u_0}{u_0}.
		\end{equation*}
		The Frobenius $\sigma$ induces an injective $\Z_q$-linear endomorphism $\sigma_1$ of $\Omega_{A^\dagger}$. Explicitly,
		\begin{equation*}
			\sigma_1 \left( a \frac{\D u_0}{u_0} \right) = \sigma(a) p \frac{\D u_0}{u_0}.
		\end{equation*}
		Thus $\sigma_1(\Omega_{A^\dagger}) = \Omega_{\sigma(A^\dagger)}$ is a free $\sigma(A^\dagger)$-module of rank $1$, generated by $p \tfrac{\D u_0}{u_0}$. By Lemma \ref{l:finiteflat}, there is a $\sigma(A^\dagger)$-linear \emph{trace map}
		\begin{equation*}
			\Tr:	A^\dagger	\to	\sigma(A^\dagger).
		\end{equation*}
		Given $a \in A^\dagger$, the trace $\Tr(a)$ may be identified with the coefficient of $X^{p-1}$ in the minimal polynomial $g_a(X)$ of $a$ over $\sigma(A^\dagger)$. Note that $g_a \equiv X^p-a^p \pmod{p}$, so that $\Tr(a)$ in fact lies in $p \sigma(A^\dagger)$. The trace map induces a $\sigma(A^\dagger)$-linear map $\Tr_1:\Omega_{A^\dagger} \to \sigma_1(\Omega_{A^\dagger})$. Explicitly,
		\begin{equation*}
			\Tr_1 \left( a \frac{\D u_0}{u_0} \right) = \Tr(a) \frac{\D u_0}{u_0} = \frac{\Tr(a)}{p} p \frac{\D u_0}{u_0}.
		\end{equation*}
		
		\begin{definition}
			The \emph{global $U_p$ operator} is the $\Z_p$-linear endomorphism
			\begin{equation*}
				U_p = \sigma_1^{-1} \circ \Tr_1: \Omega_{A^\dagger} \to \Omega_{A^\dagger}.
			\end{equation*}
		\end{definition}
		
		The global $U_p$ operator is a $p$-Dwork operator on $\Omega_{A^\dagger}$ and the iterate $U_q = U_p^{v_p(q)}$ is a $q$-Dwork operator. If we identify $\Omega_{A^\dagger}$ with $A^\dagger$ using the basis $\tfrac{\D u_0}{u_0}$, then $U_p$ is identified with the operator
		\begin{equation*}
			U_p = \frac{1}{p} \sigma^{-1} \circ \Tr:A^\dagger \to A^\dagger.
		\end{equation*}
		
	\subsection{Semi-Local Liftings}\label{ss:semilocal}
		
		For each $Q \in \{0,1,\infty\}$, the completion of $B/pB$ along $Q$ may be identified with the local field $F_Q = \F_q(\!(u_Q)\!)$. For $P\in \overline{X}$ lying above $Q$, let $F_P$ denote the completion of $A/pA$ along $P$. Then $F_P$ is a local field, and our assumptions on the branching of $\eta$ allow us to choose a uniformizer $t_P \in F_P$ such that $t_P^{e_P} = u_Q$. We have a Cartesian diagram:
		\begin{equation*}
			\begin{tikzcd}
				A/pA	\arrow[r]	&	\prod\limits_{P \in S_\eta}	F_P	\\
				B/pB	\arrow[u]	\arrow[r]	&	\prod\limits_Q	F_Q	\arrow[u]
			\end{tikzcd}.
		\end{equation*}
		
		For each $Q \in \{0,1,\infty\}$ we define $\calA_Q = \Z_q(\!(u_Q)\!)$. There is a natural map $B^\dagger \to \calA_Q^\dagger$ which is given by ``expansion in the parameter $u_Q$.'' Similarly, if $P$ is a point of $\overline{X}$ lying above $Q$, then we let $\calA_P = \Z_q(\!(t_P)\!)$. Since $A^\dagger$ is \'etale over $B^\dagger$, the natural map $A^\dagger \to F_P$ lifts uniquely to a map $A^\dagger \to \calA_P^\dagger$, which we regard as ``expansion in the parameter $t_P$.'' These maps fit into a lifted Cartesian diagram:
		\begin{equation}\label{eq:semilocal}
			\begin{tikzcd}
				A^\dagger	\arrow[r]	&	\prod\limits_{P \in S_\eta}	\calA_P^\dagger	\\
				B^\dagger	\arrow[u]	\arrow[r]	&	\prod\limits_Q	\calA_Q^\dagger	\arrow[u]
			\end{tikzcd}.
		\end{equation}
		
		For each $Q \in \{0,1,\infty\}$, the endomorphism $\sigma:B \to B$ restricts to an endomorphism of $\calA_Q$. Explicitly,
		\begin{align*}
			\sigma(u_0)	&=	u_0^p,	&	\sigma(u_\infty)	&=	u_\infty^p,	&	\sigma(u_1)	&=	(u_1+1)^p-1.
		\end{align*}
		If $P$ is a point of $\overline{X}$ lying above $Q$, then $\sigma$ extends uniquely along the \'etale map $\calA_Q^\dagger \to \calA_P^\dagger$ to an endomorphism of $\calA_P^\dagger$. Evidently, if $\eta(P) = 0$ or $\infty$ then $\sigma(t_P) = t_P^p$. The local Frobenius for $\eta(P)=1$ is more complicated:
		\begin{equation*}
			\sigma(t_P)	=	\sqrt[p-1]{\left(t_P^{p-1}+1\right)^p+1}.
		\end{equation*}
		
		\begin{definition}
			Let $\calA = \prod_{P \in S_\eta} \calA_P$. The \emph{semi-local lifting} is the ring $\calA^\dagger$ equipped with the Frobenius endomorphism $\sigma$, defined as the product of the local Frobenius endomorphisms described above.
		\end{definition}
		
		\begin{lemma}\label{l:finiteflat2}
			The Frobenius $\sigma:\calA^\dagger \to \calA^\dagger$ is injective, and $\calA^\dagger$ is a finite projective $\sigma(\calA^\dagger)$-module of rank $p$.
		\end{lemma}
		\begin{proof}
			The first statement follows exactly as in Lemma \ref{l:finiteflat}, as does the fact that $\calA^\dagger$ is flat over $\sigma(\calA^\dagger)$. It remains to show that $\calA^\dagger$ is finite over $\sigma(\calA^\dagger)$. To see this, note that if $\eta(P) = Q \in \{0,1,\infty\}$, then we have a Cartesian diagram:
			\begin{equation*}
				\begin{tikzcd}
					\sigma(\calA_P^\dagger)	\arrow[r]	&	\calA_P^\dagger	\\
					\sigma(\calA_Q^\dagger)	\arrow[u]	\arrow[r]	&	\calA_Q^\dagger	\arrow[u]	\\
				\end{tikzcd}
			\end{equation*}
			The lower horizontal map is finite of degree $p$, as can be seen explicitly from the local Frobenius structure at $Q$.
		\end{proof}
		
		As in the global situation, Lemma \ref{l:finiteflat2} implies there is a $\sigma(\calA^\dagger)$-linear \emph{trace map}
		\begin{equation*}
			\Tr:\calA^\dagger \to \sigma(\calA^\dagger).
		\end{equation*}
		As before, image of $\Tr$ lies in $p \sigma(\calA^\dagger)$. This allows us to extend $U_p$ to all of $\calA^\dagger$:
		
		\begin{definition}
			The \emph{semi-local $U_p$-operator} is defined to be the $\Z_p$-linear operator
			\begin{equation*}
				U_p	=	\frac{1}{p} \sigma^{-1} \circ \Tr:	\calA^\dagger	\to	\calA^\dagger.
			\end{equation*}
		\end{definition}
		
	\subsection{Semi-Local Decomposition}\label{ss:decomp}
	
		For local-to-global arguments, it will be convenient to have a description of the underlying $\Z_q$-module structrue of the lifting $A^\dagger$, analogous to the partial fraction decomposition of $B^\dagger$. For each point $P$ of $\overline{X}$ lying above $Q \in \{0,1,\infty\}$, we define
		\begin{equation}\label{eq:nP}
			\mu(P)	=	\begin{cases}
					0	&	Q = 0\text{ or }1	\\
					p-1	&	Q = 1
				\end{cases}.
		\end{equation}
		Consider the $\Z_q$-submodule of $\calA_P$ consisting of $\mu(P)$-truncated series
		\begin{equation}\label{eq:NP}
			\calA_P^\tr	=	\left\{	\sum_k a_k t^{-k} \in \calA_P	:	a_k = 0\text{ for all }k \leq \mu(P)	\right\}.
		\end{equation}
		We define $\calA^\tr = \bigoplus_P \calA_P^\tr$, regarded as a $\Z_q$-submodule of $\calA$. There is a natural projection map $\pr:\calA \to \calA^\tr$. Since $\calA^\tr$ is a $p$-saturated submodule of $\calA$, the $p$-adic completion $\calA^{\infty,\tr}$ is naturally a $\Z_q$-submodule of $\calA^\infty$. We define $\calA^{\dagger,\tr} = \calA^\dagger \cap \calA^{\infty,\tr}$. 
		
		\begin{proposition}\label{p:semilocal}
			There is an exact sequence of $\Z_q$-modules
			\begin{equation*}
				0	\to	L	\to	A^\dagger	\xrightarrow{\pr}	\calA^{\dagger,\tr}	\to	0,
			\end{equation*}
			where $L$ is a finite \emph{free} $\Z_q$-module of rank
			\begin{equation}\label{eq:rank}
				N	=	g-1+r_0+r_1+r_\infty.
			\end{equation}
		\end{proposition}
		\begin{proof}
			First, consider the map $\pr:A^\infty \to \calA^{\infty,\tr}$. Any element of the kernel $L$ has poles of finite order about each $P$, so that $L \subseteq A^\dagger$. Consider the reduction modulo $p$:
			\begin{equation*}
				\overline{L}	\to	\overline{A}	\xrightarrow{\overline{\pr}}	\overline{\calA}^\tr.
			\end{equation*}
			By the Riemann-Roch theorem, $\overline{\pr}$ is surjective. The kernel is precisely the global sections of the line bundle $\mathcal{L}(D)$, where $D = \sum_{\eta(P)=1}(p-1)P$. Again by Riemann-Roch, this is a finite $\F_q$-module of rank $N$. By Lemma \ref{l:red}(\ref{i:surj})-(\ref{i:ker}), the map $\pr:A^\infty \to \calA^{\infty,\tr}$ is surjective and $\ker(\overline{\pr})=\overline{L}$. Since $L$ is $p$-torsion free, Nakayama's lemma implies that $L$ is a finite free $\Z_q$-module of rank $N$.
			
			It remains to show that $\pr:A^\dagger \to \calA^{\dagger,\tr}$ is surjective. Observe that $a \in \calA^{\infty,\tr}$ lies in $\calA^{\dagger,\tr}$ if and only if there exists $\hat{a} \in \calA^\dagger$ such that $\pr(\hat{a}) = a$. By the preceding discussion, we may choose $\hat{a}$ to lie in the image of $A^\infty$. Since $A^\dagger = \calA^\dagger \cap A^\infty$, this completes the proof. 
		\end{proof}

\section{\texorpdfstring{$\sigma$}{Sigma}-Modules}\label{s:trace}
	
	\subsection{Definitions}\label{ss:sigmamod}
	
		For this subsection, let $X$ denote a general affine $\F_p$-scheme. By a \emph{flat lifting} of $X$ over $R$, we will mean a pair $(A,\sigma)$, where $A$ is a lifting of the coordinate ring of $X$ to a flat $R$-algebra, and $\sigma:A \to A$ is a lifting of the absolute Frobenius endomorphism of $X$.
		
		\begin{definition}
			A \emph{$\sigma$-module} over $A$ is a pair $(M,\phi)$, where $M$ is a finite projective $A$-module and $\phi:\sigma^*M \to M$ is an $R$-linear map such that $K \otimes \phi$ is an isomorphism. We say that $(M,\phi)$ is \emph{unit-root} if $\phi$ is an isomorphism.
		\end{definition}
	
		Let $(M,\phi)$ be a $\sigma$-module over $A$. Let $(B,\tau)$ be a flat lifting of another affine $\F_p$-scheme $Y$. Given a \emph{Frobenius-compatible} map $f:A \to B$ of $R$-algebras, we obtain by extension of scalars a $\tau$-module $(B \otimes_A M,\tau \otimes \phi)$ over $B$.
		
		\begin{theorem}[\cite{Katz}, 4.1.1]\label{t:Katz}
			There is a rank-preserving equivalence between the category of unit-root $\sigma$-modules over $A^\infty$ and the category of $R$-valued representations of $\pi_1(X)$. Under this equivalence, the pullback of representations along $Y \to X$ corresponds to extension of scalars along the Frobenius-compatible map $f:A^\infty \to B^\infty$.
		\end{theorem}
		
		Let $(M,\phi)$ be a $\sigma$-module over $A$. By composition with the natural map $M \to \sigma^* M$, we will usually regard $\phi$ as an endomorphism of $M$ which is $\sigma$-semilinear, in the sense that $\phi(a m) = \sigma(a) \phi(m)$ for all $a \in A$ and all $m \in M$. We will typically restrict our attention to the case that $M$ is a \emph{free} $A$-module of rank $1$. If $e \in M$ is any basis for $M$ and $E = \phi(e)$ is the ``matrix'' of $\phi$ with respect to $E$, then we have an isomorphism of $\sigma$-modules
		\begin{equation*}
			(M,\phi)	\cong	(A,E \circ \sigma),
		\end{equation*}
		In this case, we refer to $E$ as a \emph{Frobenius structure} for $(M,\phi)$.
		
		For each $x \in |X|$, let $\deg(x) = [k(x):\F_q]$. Consider the flat lifting $(R(x),F)$ of $k(x)$ over $R$, where $R(x) =  R \otimes_{\Z_p} W(k(x))$ and $F$ is the canonical $R$-linear lifting of Frobenius. By \cite[Theorem 3.3]{Monsky} there is a unique Frobenius-compatible map
		\begin{equation*}
			\hat{x}:	A	\to	R(x),
		\end{equation*}
		which we call the \emph{Teichm\"uller lifting} of $x$. Let $(M_x,\phi_x)$ denote the $F$-module over $R(x)$ obtained by extension of scalars along $\hat{x}$. The map $\phi_x$ is only $R$-linear, but the iterate $\phi_x^{v_p(q)\deg(x)}$ is $R(x)$-linear. Since $R(x)$ is local, $M_x$ is free of finite rank and we may define:
		
		\begin{definition}
			The \emph{$L$-function} of $(M,\phi)$ is the power series
			\begin{equation*}
				L(\phi,s)	=	\prod_{x \in |U|} \frac{1}{\det(1-\phi_x^{v_p(q)\deg(x)} s^{\deg(x)})} \in	1+sR[[s]].
			\end{equation*}
		\end{definition}
		
	\subsection{\texorpdfstring{$\sigma$}{Sigma}-Modules over the Global Lifting}
	
		We will now return to the setting of \S \ref{s:geo}. Let $(A^\dagger,\sigma)$ denote the global lifting of $U$ as defined in \S \ref{ss:global}. Let $A_\pi = R \otimes_{\Z_p} A$, which we equip with the $\pi$-adic topology. By weak base change, we obtain a flat lifting $(A_\pi^\dagger,\sigma)$ of $U$ over $R$. Let $\rho:\pi_1(X) \to R^\times$ be a continuous character. Let $\rho_U$ denote the restriction of $\rho$ to $\pi_1(U)$, and let $(M,\phi)$ be the corresponding unit-root $\sigma$-module over $A_\pi^\infty$. The $L$-function $L(\phi,s)$ coincides with the Artin $L$-function $L(\rho_U,s)$ of $\rho_U$ over $U$. In other words,
		\begin{equation}\label{eq:ArtinLFunction}
			\prod_{\substack{x \in |X - U|\\x \notin S}} (1-\rho(\Frob_x) s)\cdot L(\rho,s)  = L(\phi,s).
		\end{equation}
	
		\begin{definition}\label{d:oc}
			We say that $(M,\phi)$ is \emph{overconvergent} if $(M,\phi)$ is obtained by extension of scalars from a $\sigma$-module over $A_\pi^\dagger$.
		\end{definition}
		
		If $(M,\phi)$ is overconvergent, then $L(\phi,s)$ is known to be meromorphic by the Monsky trace formula. Let us state the trace formula in the special case that $M$ is free of rank $1$. First, observe that $(M,\phi)$ is overconvergent if and only if $(M,\phi)$ admits an overconvergent Frobenius structure $E \in A_\pi^\dagger$. By weak base change, the operator $U_p$ induces an $R$-linear endomorphism $U_p:A_\pi^\dagger \to A_\pi^\dagger$. Upon fixing a choice of $E$, we define ``Dwork operator''
	 	\begin{equation*}
	 		\Theta = U_p \circ E:A_\pi^\dagger \to A_\pi^\dagger.
	 	\end{equation*}
	 	Let $V_\pi = K \otimes_R A_\pi$, regarded as a vector space over $K_q$. By \cite[Theorem 2.1]{Monsky}, the action of $\Theta_q$ on $V_\pi^\dagger$ is nuclear. We can now state Monsky's trace formula:
		
		\begin{theorem}[\cite{Monsky}, or \cite{Taguchi} when $\mathrm{char}(R) = p$]\label{t:Monsky}
			Suppose that $(M,\phi)$ is overconvergent and that $M$ is free of rank $1$. Then
			\begin{equation*}
				L(\phi,s)	=	\frac{C(\Theta_q|V_\pi^\dagger,s)}{C(\Theta_q|V_\pi^\dagger,qs)}.
			\end{equation*}
		\end{theorem}
		
		\begin{corollary}
			\label{c:Monsky for representations}
			Let $\rho:\pi_1(X) \to R^\times$ be a continuous character and let $(M,\phi)$ be the
			unit-root $\sigma$-module over $A_\pi^\dagger$ corresponding to $\rho_U$. Assume that $(M,\phi)$ is overconvergent and that $M$ is free of rank $1$. Then
			\begin{equation*}
			\prod_{\substack{x \in |X - U|\\x \notin S}} (1-\rho(\Frob_x) s)\cdot L(\rho,s)  =\frac{C(\Theta_q|V_\pi^\dagger,s)}{C(\Theta_q|V_\pi^\dagger,qs)}.
			\end{equation*}
		\end{corollary}
		
		\begin{corollary}
			In the setting of Corollary \ref{c:Monsky for representations}, $L(\rho,s)$ is analytic in the disk $v_\pi(s) > -v_\pi(q)$.
		\end{corollary}
		\begin{proof}
			Note that $L(\phi,s)$ is analytic in this disk by the Monsky trace formula. By a theorem of Crew \cite{Crew}, $L(\rho,s)$ has no poles in the region $v_\pi(s) = 0$. The claim follows from Corollary \ref{c:Monsky for representations}.
		\end{proof}
		
	\subsection{Semi-Local Growth Conditions}
	
		Let $P$ be a branch point of the covering $\eta$ defined in \S \ref{ss:mapping}. Recall that we have defined a flat lifting $(\calA_P^\dagger,\sigma)$ of the local field $F_P$ over $\Z_p$ in \S \ref{ss:semilocal}. Let $\calA_{\pi,P} = R \otimes_{\Z_p} \calA_P$, equipped with the $\pi$-adic topology. Again by weak base change, we obtain a flat lifting $(\calA_{\pi,P}^\dagger,\sigma)$ of $F_P$ over $R$. Let $\rho:\pi_1(X) \to R^\times$ be a continuous character, and let $(M,\phi)$ be unit-root $\sigma$-module over $A_\pi^\infty$ corresponding to $\rho_U$. By extension of scalars along the Frobenius-compatible map $A_\pi^\infty \to \calA_{\pi,P}^\infty$, we obtain a $\sigma$-module $(M_P,\phi_P)$ over $\calA_{\pi,P}^\infty$.
		
		\begin{lemma}\label{l:unr}
			If $\rho$ is unramified at $P$, then $(M_P,\phi_P)$ admits a constant Frobenius structure $\tilde{E}_P \in R_q$.
		\end{lemma}
		\begin{proof}
			Since $(M,\phi)$ is unramified, $\rho$ factors through a character $\overline{\rho}:\pi_1(\Spec(\F_q)) \to R^\times$. The latter corresponds to a unit-root $F$-module $(\overline{M},\overline{\phi})$ over $R_q$. By Theorem \ref{t:Katz}, $(M,\phi)$ is obtained by extension of scalars from $(\overline{M},\overline{\phi})$ along $R_q \to \calA_{\pi,P}$, so $(M,\phi)$ must admit a Frobenius structure in $R_q$.
		\end{proof}
		
		We define as before $\calA_\pi = \prod_{P \in S_\eta} \calA_{\pi,P}$. By a \emph{semi-local Frobenius structure} for $(M,\phi)$, we will mean an element $\tilde{E} = (\tilde{E}_P) \in \calA_\pi^\infty$ such that $\tilde{E}_P$ is a Frobenius structure for $(M_P,\phi_P)$ for all $P \in S_\eta$.

		\begin{definition}
			Let $\bm\delta = (\delta_P)$ be a tuple of positive rational numbers indexed by $S$. We say that $\rho$ is $\pi$-adically \emph{$\bm\delta$-overconvergent} if $(M,\phi)$ admits a semi-local Frobenius structure $\tilde{E}$ such that:
			\begin{enumerate}
				\item	If $P \notin S$, then $\tilde{E}_P \in 1+\pi R_q$ is constant. In particular, $(M,\phi)$ is unramified at $P$.
				\item	If $P \in S$, then $\tilde{E}_P = \tilde{E}_P \in R_q[t_P^{-1}]^{\delta_P} \cap (1+\pi R_q\langle t_P^{-1}\rangle)$.
			\end{enumerate}
			In this case we refer to $\tilde{E}$ as a \emph{$\bm\delta$-Frobenius structure} for $(M,\phi)$.
		\end{definition}
		
		\begin{example}\label{ex:deltaOverconvergent}
			\begin{enumerate}
				\item	Suppose that $\rho:\pi_1(X) \to R^\times$ is finite of order $p^n$. Let $\pi \in \Z_p[\rho]$ be a uniformizer. For each $P \in S$, let $d_P$ denote the Swan conductor of $\rho$ at $P$, and let $\delta_P = d_P/p^{n-1}$. By \cite[Proposition 5.5]{KramerMiller}, $\rho$ is $\pi$-adically $\bm\delta$-overconvergent.
				\item	Let $X_\infty/X$ be a $\Z_p$-tower of curves (\S\ref{ss:stable}), regarded as a surjective map
				\begin{equation*}
					\rho:\pi_1(X) \to \Z_p.
				\end{equation*}
				Let $\chi: \Z_p \to R^\times$ be a continuous character, and let $\rho_\chi = \chi \circ \rho$. We define $\pi_\chi = \chi(1)-1$. In \cite{Kramer-Miller-Upton2} we prove that if $X_\infty/X$ has $\bm\delta$-stable monodromy and $R$ has characteristic $p$, then $\rho_\chi$ is $\pi_\chi$-adically $\bm\delta$-overconvergent.
				\item	Continuing the previous example: Let $X_\infty/X$ be an overconvergent $\Z_p$-tower of curves (\ref{ss:oc}), so that $X_\infty/X$ has $\bm\delta$-stable monodromy for some $\bm\delta = (\delta_P)_{P \in S}$. In \cite{Upton2}, we show that for $v_{\pi_\chi}(p) \gg 0$,  $\rho_\chi$ is $\pi_\chi$-adically $\bm\delta$-overconvergent.
			\end{enumerate}
		\end{example}

		\subsection{Global growth conditions}
		The Monsky trace formula requires that the $\sigma$-module $(M,\phi)$ in question is overconvergent (i.e. it has
		an overconvergent Frobenius structure). The following proposition says that one may check the overconvergence of
		$M$ locally at each $P \in S$. In particular, the Monsky trace formula may be applied to any $\sigma$-module that 
		is $\bm\delta$-overconvergent.
		
		\begin{proposition}\label{p:oc}
			Let $(M,\phi)$ be a $\sigma$-module over $A_\pi^\infty$ with Frobenius structure $E \in 1 + \pi A_\pi^\infty$. 
			Suppose that $(M,\phi)$ admits an overconvergent semi-local Frobenius structure $\tilde{E} \in 1+\pi \calA_\pi^\dagger$. Then
			there exists $a \in 1 + \pi A_\pi^\infty $ such that $E'= \sigma(a) E a^{-1}$ is contained in $1+\pi A_\pi^\dagger$.
			In particular $(M,\phi)$ is overconvergent.
		\end{proposition}
		We define the $k$-th partial valuation 
		on $\mathcal{A}_{\pi,P}^\infty$ as follows:
		For $f \in \mathcal{A}_{\pi,P}^\infty$ write $f = \sum a_n t_P^{n}$. We define $w_k^{P}(f)$ to be the
		$t_P$-adic order of $f$ reduced modulo $\pi^{k+1}$. That is, 
		\begin{align*}
		w_k^{P}(f) &= \min_{v_\pi(a_n)\leq k} n.
		\end{align*}
		Note that we have 
		\begin{align}\label{eq:partial valuation multiplication inequality}
		w_k^{P}(fg) &\geq \min_{i+j=k} w_i^P(f) + w_j^P(g)
		\end{align}
		
		\begin{lemma}\label{l:frobenius multiplies valuation by p}
			Let $m$ be an integer large enough so that $\sigma(t_P)\in \mathcal{A}_{\pi,P}^{pm}(p)$ (note that by our definition of $\sigma$ and $t_P$, such an $m$ always exists). Let $x \in \mathcal{A}_{\pi,P}^\infty$. Fix a natural number $k_0$. Assume that $w^P_{k}(x)>-km$ for
			$k<k_0$ and $w^P_{k_0}(x)< -k_0m$. Then $w^P_{k_0}(\sigma(x))=pw^P_{k_0}(x)$ and $w^P_k(\sigma(x))>-kpm$ for $k < k_0$.
		\end{lemma}
		\begin{proof}
			This is deduced from \eqref{eq:partial valuation multiplication inequality}.
		\end{proof}

		\begin{lemma} \label{l: how to identify OC Frobenius structures}
			Let $E_P'$ be a Frobenius structure of $(M_P, \phi_P)$ with $E_P' \equiv 1 \mod \pi$ and let $C>0$. Assume that for $k\gg 0$ we have either
			$w_k^{P}(E_P') \in \Z \backslash p\Z$ or $w_k^{P}(E_P') \geq -C$. Then $E_P'$ is overconvergent. 
		\end{lemma}
		\begin{proof}
			Let $\tilde{E}_P$ be the $P$-th coordinate of $\tilde{E}$. In particular, $\tilde{E}_P$ is a Frobenius structure of $(M_P,\phi_P)$ contained in $\mathcal{A}^\dagger_{\pi,P}$
			with $\tilde{E}_P \equiv 1\mod \pi$. Then we have
			$\tilde{E}_P \frac{\sigma(b)}{b} = E_P'$ for some $b \in \mathcal{A}_{\pi,P}^\infty $ with $b \equiv 1 \mod \pi$. We need to prove $b$ is overconvergent.
			Assume the contrary. In terms of partial valuations, this means that for every $m>0$ there exists $k_m$ such that
			$w_{k_m}^{P}(b)< -mk_m$.
			Assume that $k_m$ is the smallest integer such that this inequality holds (so that $w_{k}^{P}(b)\geq  -mk$ for $k\leq k_m$). 
			If $m$ is sufficiently large, we know from \eqref{eq:partial valuation multiplication inequality} and Lemma \ref{l:frobenius multiplies valuation by p} that 
			\begin{align}\label{eq:growth of base change equations}
			\begin{split}
			w_{k_m}^{P}\Big (\frac{\sigma(b)}{b}\Big) &= pw^P_{k_m}(b)<-pmk_m, \\
			w_{k}^{P}\Big(\frac{\sigma(b)}{b}\Big) &> -pmk, \text{ for $k<k_m$.}
			\end{split}
			\end{align}
			Since $\tilde{E}_P$ is overconvergent, there exists $m_0>0$ such that $w_k^{P}(\tilde{E}_P)\geq -m_0k$ for all $k\geq 0$. Then from \eqref{eq:partial valuation multiplication inequality} and \eqref{eq:growth of base change equations} 
			we have
			\begin{align*}
			w_{k_m}^{P}(E_P') &= pw^P_{k_m}(b),
			\end{align*}
			for $m$ sufficiently large.
			If we take $m$ to be larger than $C$, this contradicts our assumption on the partial
			valuations of $E_P'$. 
		\end{proof}

		\begin{proof}(Of proposition \ref{p:oc})
			Recall that $E \in A_\pi^\infty$ be a Frobenius structure of $(M,\phi)$ with $E \equiv 1 \mod \pi$. 
			For each $P$, let $A_{\pi,P}$ denote the subring of $A_\pi$ whose only poles are at $[P]$, the Teichmuller lift of $P$. Then $\Spec(A_{\pi,P})$ is an affine curve over whose special fiber is
			$\overline{X}-P$. We will show that there exists $a_P \in A_{\pi,P}^\infty$ such that $E\frac{\sigma(a_P)}{a_P}$ is 
			overconvergent at $P$ (i.e. it lies in $\mathcal{A}_{\pi,P}^\dagger $). Note that if $E$ is overconvergent at $Q \neq P$, then so is $E\frac{\sigma(a_P)}{a_P}$,
			since $a_P$ does not have a pole at $Q$. The proposition will follow by repeating for each $P$.
			
			Let $C$ be sufficiently large so that $H^0(\overline{X}, \mathcal{O}_{\overline{X}}((C+n)P))/H^0(\overline{X},\mathcal{O}_{\overline{X}}(CP))$ has dimension $n$.
			We will find $a_P \in A_{\pi,P}^\infty$ such that $E\frac{\sigma(a_P)}{a_P}$ satisfies the hypothesis of Lemma \ref{l: how to identify OC Frobenius structures}. 
			More precisely, we will inductively construct $a_{P,n} \in A_{\pi,P}^\infty$ satisfying
			\begin{enumerate}
				\item $a_{P,n+1} \equiv a_{P,n} \mod \pi^{n+1}$
				\item We can write $E \frac{\sigma(a_{P,n})}{a_{P,n}} \equiv b_{P,n} + c_{P,n} \mod \pi^{n+1}$ where $c_{P,n} \in t_P^{-C}R\llbracket t_P \rrbracket$ and $b_{P,n}$ is of the form
				\begin{align*}
				b_{P,n} &= \sum_{p\nmid k} x_{P,n,k} t_P^{k}.
				\end{align*}
			\end{enumerate}		
			We then take $a_P= \lim a_{P,n}$ and the result follows from Lemma \ref{l: how to identify OC Frobenius structures}.
			
			For $n=0$ we take $a_{P,n}=1$. Let $n>0$ and assume such an $a_{P,n}$ exists. We write 
			\begin{align*}
			E \frac{\sigma(a_{P,n})}{a_{P,n}} &= b_{P,n} + c_{P,n} + \pi^{n+1}r_n \mod \pi^{n+2}.
			\end{align*}
			Let $r_{n,0}$ denote the reduction of $r_n$ modulo $\pi$. Break up the $t_P$-adic expansion of $r_{n,0}$:
			\begin{align*}
			r_{n,0} = \underbrace{  \sum_{k= -pC} ^\infty y_{0,k} t_P^k}_{\alpha_0} + \underbrace{\sum_{\stackrel{k<-pC}{p\nmid k}}y_{0,k} t_P^k}_{\beta_0} + \underbrace{\sum_{pk<-pC} y_{0,pk} t_P^{pk}}_{\gamma_0}.
			\end{align*}
			By our assumption on $C$, there exists a regular function $c_0$ on $\overline{X}-P$ with
			\begin{align*}
			c_0^p &\equiv \gamma_0 \mod t_P^{-C}\mathbb{F}_q  \llbracket t_P \rrbracket.
			\end{align*}
			We then set $r_{n,1}=r_{n,0}-c_0^p+c_0$. Break up $r_{n,1}$ into $\alpha_1 + \beta_1 + \gamma_1$ as above and we find $c_1$ with $c_1^p \equiv \gamma_1 \mod t_P^{-C}\mathbb{F}_q  \llbracket t_P \rrbracket$. We define $r_{n,2}=c_1^p-c_1$. Repeat this process. Note that the order of the pole of $\gamma_i$ decreases. In particular, for $i$ large enough we have $\gamma_i=0$. Write $c=\sum c_i$ and let
			$\tilde{c}$ be a lift of $c$ to $A_{\pi,P}$. By definition we have
			\begin{align*}
			r_{n,0} - c^p + c &= \sum_{k= -pC} ^\infty y_k t_P^k + \sum_{\stackrel{k<-pC}{p\nmid k}}y_{k} t_P^k.
			\end{align*}
			We define $a_{P,n+1}=a_{P,n}(1-\pi^{n+1}\tilde{c})$, which satisfies the correct properties
		\end{proof}

	\subsection{Semi-Local Twisting}\label{ss:twist}
	
		Henceforth we will assume that $\rho:\pi_1(X) \to R^\times$ is $\bm\delta$-overconvergent, and let $(M,\phi)$ denote the unit-root $\sigma$-module over $A_\pi^\dagger$ corresponding to $\rho$. We fix a $\bm\delta$-Frobenius structure $\tilde{E}$ for $(M,\phi)$ and a global Frobenius structure $E = \sigma(a) \tilde{E} a^{-1} \in A_\pi^\dagger$ as in Proposition \ref{p:oc}. By weak base change, the semi-local $U_p$ operator (\ref{ss:semilocal}) induces a $p$-Dwork operator $U_p:\calA_\pi^\dagger \to \calA_\pi^\dagger$. We define $\tilde{\Theta}$ to be the $p$-Dwork operator
	 	\begin{equation*}
	 		\tilde{\Theta} = U_p \circ \tilde{E}:\calA_\pi^\dagger \to \calA_\pi^\dagger.
	 	\end{equation*}
		Since $\Theta$ is $\sigma^{-1}$-linear, we have
		\begin{align*}
			\Theta	=	U_p	\circ E = U_p \circ (\sigma(a) \circ \tilde{E} \circ a^{-1})	=	a \circ (U_p \circ \tilde{E}) \circ a^{-1}.
		\end{align*}
		Let $\tilde{A}_\pi^\dagger = a^{-1} \cdot A_\pi^\dagger$ and $\tilde{V}_\pi^\dagger = K \otimes_R \tilde{A}_\pi^\dagger$, regarded as a $K_q$-vector space. The following is now immediate:
		
		\begin{proposition}
			The operator $\tilde{\Theta}$ restricts to an endomorphism of $\tilde{V}_\pi^\dagger$. Moreover, the action of $\tilde{\Theta}_q$ on this space is nuclear, and there is an equality of Fredholm series
			\begin{equation*}
				C(\tilde{\Theta}_q|\tilde{V}_\pi^\dagger,s) = C(\Theta_q|V_\pi^\dagger,s).
			\end{equation*}
		\end{proposition}
		
		For each point $P$ of $\overline{X}$ with  $\eta(P) \in  \{0,1,\infty\}$, let $\calA_{\pi,P}^\tr = R \otimes_{\Z_p} \calA_P^\tr$, regarded as an $R_q$-submodule of $\calA_{\pi,P}$. We define as before $\calA_\pi^\tr = \bigoplus_P \calA_{\pi,P}^\tr$. As in Proposition \ref{p:semilocal}, we see that there is an exact sequence of $R_q$-modules
		\begin{equation*}
			0	\to	L_\pi	\to	A_\pi^\dagger	\to	\calA_\pi^{\dagger,\tr}	\to	0,
		\end{equation*}
		where $L_\pi$ is a finite free $R_q$-module of rank $N$ (\ref{eq:rank}). The following lemma states that the ``twisted'' space $\tilde{A}_\pi^\dagger$ can be decomposed in a similar way:
		
		\begin{lemma}
			There is an exact sequence of $R_q$-modules
			\begin{equation*}
				0	\to	\tilde{L}_\pi	\to	\tilde{A}_\pi^\dagger	\to	\calA_\pi^{\dagger,\tr}	\to	0,
			\end{equation*}
			where $\tilde{L}_\pi = a^{-1} L_\pi$ is a finite free $R_q$-module of rank $N$.
		\end{lemma}
		\begin{proof}
			Indeed, since $a \equiv 1 \pmod{\pi}$ this sequence reduces modulo $\pi$ to the exact sequence described above. The result follows from Lemma \ref{l:red}.
		\end{proof}
	
\section{Estimating Dwork Operators}\label{s:estimates}

	In this section we give explicit estimates for the Newton polygon of $\tilde{\Theta}_q$ acting on the twisted space $\tilde{V}_\pi^\dagger$.
	
	\subsection{Local Estimates}\label{s:localestimates}
	
		We begin by studying the action of $\tilde{\Theta}$ on the local spaces $\calV_{\pi,P}^{\dagger} = K \otimes_R \calA_{\pi,P}^{\dagger}$. Our approach will be to restrict $\tilde{\Theta}$ to a Banach subspace $\calV_{\pi,P}^{\mathbf{m}} \subset \calV_{\pi,P}^{\dagger}$, which is defined in terms of certain local growth conditions at $P$. In \S\ref{ss:case1}-\ref{ss:case2} below we estimate the action of $\tilde{\Theta}$ on this subspace.
		
		First, let us explain our notational conventions. The superscript $\mathbf{m} = (m_P)$ will denote a tuple of positive rational numbers indexed by $S_\eta$. Since the Newton polygon of $\tilde{\Theta}$ remains the same under extension of scalars, it will be convenient to assume that $R$ contains an $m_P$-th root of $\pi$ for each $P\in S_\eta$. We will define a subspace $\calA_{\pi,P}^\mathbf{m}\subset \calA_{\pi,P}^\dagger$ using certain growth conditions depending on $m_P$ and $\eta(P)$. For $\eta(P) = 0$ or $\infty$, the definition is simple:
		\begin{equation*}
			\calA_{\pi,P}^\mathbf{m} = \calA_{\pi,P}^{m_P} \subseteq \calA_{\pi,P}^\dagger.
		\end{equation*}
		When $\eta(P) = 1$, the definition of $\calA_{\pi,P}^\mathbf{m}$ is more complicated and will be given in \S\ref{ss:case2}.
		
		Our present goal is to estimate the columns of the matrix of $\tilde{\Theta}$ with respect to a formal basis of $\calA_{\pi,P}^{\mathbf{m},\tr} = \calA_{\pi,P}^\mathbf{m} \cap \calA_{\pi,P}^{\dagger,\tr}$, at least for suitable $\mathbf{m}$. To explain the condition on $\mathbf{m}$, consider the tuple of rational numbers $\mathbf{m}_\pi = (m_{\pi,P})_{P \in S_\eta}$ defined by:
		\begin{equation*}
			m_{\pi,P} = \begin{cases}
				\frac{\delta_P}{p} & P \in S \\
				0 & P \not\in S \text{ and } \eta(P) \in \{0,\infty\} \\
				\frac{1}{v_\pi(p)} &  \eta(P) = 1 
			\end{cases}.
		\end{equation*}
		We will write $\mathbf{m} \geq \mathbf{m}_\pi$ if $m_P \geq m_{\pi,P}$ for all $P \in S_\eta$. We will give our estimates for all tuples $\mathbf{m}$ of \emph{positive} rational numbers with $\mathbf{m} \geq \mathbf{m}_\pi$.
		
	\subsubsection{First case: \texorpdfstring{$\eta(P) = 0$ or $\infty$}{eta(P)=0 or infinity}}\label{ss:case1}
	
		We first consider the case $\eta(P) = 0$ or $\infty$. In this case, we have already defined $\calA_{\pi,P}^\mathbf{m}$ to be $\calA_{\pi,P}^{m_P}$. We have $\sigma(t_P) = t_P^p$, and the operater $U_p$ is given explicitly as follows: Let $k = p\ell+r$ with $0 \leq r < p$. Then
		\begin{equation*}
			U_p(t_P^k) = \begin{cases}
					t_P^\ell	&	r = 0	\\
					0	&	r \neq 0
				\end{cases}.
		\end{equation*}
		From the above formula we see that for every $m > 0$,
		\begin{equation}\label{eq: Up estimate}
			U_p(\calA_{\pi,P}^{m}) \subseteq \calA_{\pi,P}^{m/p}.
		\end{equation}
		
		\begin{proposition}\label{p:case1}
			Suppose that $\eta(P) = 0$ or $\infty$, and that $\mathbf{m} \geq \mathbf{m}_\pi$. Then
			\begin{equation*}
				\tilde{\Theta}(\pi^{k/m_P}t_P^{-k}) \in \pi^{\frac{k(p-1)}{pm_P}} \calA_{\pi,P}^{\mathbf{m}}.
			\end{equation*}
		\end{proposition}
		\begin{proof}
			If $P \notin S$, then $\tilde{E}_P \in 1+\pi R_q$ is constant. If $P \in S$, then $\tilde{E}_P \in \calA_{\pi,P}^{\delta_P}$ since $(M,\phi)$ is $\bm\delta$-overconvergent. In either case, we have $\tilde{E}_P \in \calA_{\pi,P}^{p m_P}$. By definition, we know that $\pi^{k/m_P}t_P^{-k}$ is contained in $\pi^{\frac{k(p-1)}{pm_P}} \calA_{\pi,P}^{pm_P}$. The result follows from \eqref{eq: Up estimate}.
		\end{proof}
		
	
	\subsubsection{Second Case: \texorpdfstring{$\eta(P) = 1$}{eta(P)=1}}\label{ss:case2}
	
		We now turn to the case $\eta(P) = 1$. Since $P \notin S$, $\tilde{E}_P$ is constant. Thus we will focus on estimating the action of the operator $U_p$. Let $u_P = t_P^{p-1}$ and let $\calB_{\pi,P} = R_q(\! (u_P)\! )$. Then $\calA_{\pi,P}/\calB_{\pi,P}$ is finite and Galois so we have a decomposition
		\begin{equation*}
			\calA_{\pi,P}	= \bigoplus_{i = 0}^{p-1} t_P^{-i}\calB_{\pi,P}
		\end{equation*}
		into eigenspaces for the distinct characters of the Galois group. The induced splitting of $\calA_{\pi,P}^\dagger$ is $\sigma$-equivariant. Consider the subspace:
		\begin{equation*}
			\calB_{\pi,P}^\mathbf{m}	=	 \calB_{\pi,P}^{m_P} = \left\{	\sum_{k = -\infty}^\infty b_k u^{-k}: v_\pi(b_k) >  \frac{k}{m_P}	\text{ for all }k>0	\right\}.
		\end{equation*}

		\begin{lemma}\label{l:km Up lemma}
			Suppose that $\mathbf{m} \geq \mathbf{m}_\pi$. Let $k = p\ell+r \in \Z$ with $0 \leq r < p$. Then
			\begin{equation*}
				U_p(t^{-k}) \in t^{-(\ell+r)}\calB_{\pi,P}^\mathbf{m}.
			\end{equation*}
		\end{lemma}
		\begin{proof}
			For $R=\Z_p$ and $\pi=p$, this is {\cite[Corollary 4.7]{KramerMiller}}. The condition $\mathbf{m} \geq \mathbf{m}_\pi$ guarantees that $m_P \geq \tfrac{1}{v_\pi(p)}$, so the general result follows by tensoring up to $R$.
		\end{proof}
		
		\begin{definition}
			We define
			\begin{equation*}
				\calA_{\pi,P}^\mathbf{m} = \bigoplus_{i = 0}^{p-1} t_P^{-i} \calB_{\pi,P}^\mathbf{m}.
			\end{equation*}
		\end{definition}
		
		For each $k > p-1$, define the positive integer
		\begin{equation*}
			a(k) =  \left\lfloor \frac{k-1}{p-1} \right\rfloor.
		\end{equation*}
		Then the elements $\pi^{a(k)/m_P} t_P^{-k}$ constitute a formal basis for $\calA_{\pi,P}^\mathbf{m}$. We can now give our local estimate for $\tilde{\Theta}$ in this case:

		\begin{proposition}\label{p:case2}
			Suppose $\eta(P) = 1$ and that  $\mathbf{m} \geq \mathbf{m}_\pi$. Then for each $k = p\ell+r > p-1$, we have
			\begin{equation*}
				\tilde{\Theta}(\pi^{a(k)/m_P}t_P^{-k}) \in \pi^{\ell/m_P} \calA_{\pi,P}^\mathbf{m}.
			\end{equation*}
		\end{proposition}
		\begin{proof}
			Observe that $a(k) - a(\ell+r) = \ell$. Thus, from Lemma \ref{l:km Up lemma} we have 
			\begin{equation*}
				U_p(\pi^{a(k)/m_P}t_P^{-k}) = \pi^{\ell/m_P} \pi^{a(\ell+r)/m_P} t^{-(\ell+r)} \calB_{\pi,P}^\mathbf{m} \subseteq \pi^{\ell/m_P}\calA_{\pi,P}^\mathbf{m}.
			\end{equation*}
			The claim follows since $\tilde{E}_P \in 1+\pi R_q$.
		\end{proof}
		
		\begin{remark}\label{r:2}
			Suppose that $p = 2$. For $k \geq 3$, define $a(k) = \lfloor (k-1)/3 \rfloor$. A similar construction provides a submodule $ \calA_{\pi,P}^\mathbf{m} \subseteq  \calA_{\pi,P}^\dagger$ with the following property: Let $k = 2\ell - r$ with $r = 0$ or $1$. Then
			\begin{equation*}
				U_p(\pi^{a(k)/m_P}t_P^{-k}) \in \pi^{(a(k)-a(\ell+r))/m_P} \calA_{\pi,P}^\mathbf{m}.
			\end{equation*}
			This estimate is too low for applications to the global setting. For example, if $k = 5 = 2\cdot 3-1$, then $a(k) - a(\ell+r) = 0$, and this contributes an extra segment of slope $0$ in the global Hodge bound below.
		\end{remark}
		
	\subsection{Global Estimates}\label{ss:globalestimates}

		Let $\mathbf{m} \geq \mathbf{m}_\pi$, and consider the ``twisted'' spaces $\tilde{A}_\pi^{\mathbf{m}} = \calA_\pi^{\mathbf{m}} \cap \tilde{A}_\pi^\dagger$ and  $\tilde{V}_\pi^{\mathbf{m}}=K \otimes_R\tilde{A}_\pi^{\mathbf{m}}$. We will now define a formal basis for $\tilde{V}_\pi^{\mathbf{m}}$ by lifting the local bases used in \S \ref{s:localestimates}. Let $\calA_\pi^{\mathbf{m}}$ denote the product of the $\calA_{\pi,P}^{\mathbf{m}}$ defined in \S\ref{s:localestimates}. Define $\calA_\pi^{{\mathbf{m}},\tr} = \calA_\pi^{\mathbf{m}} \cap \calA_\pi^{\dagger,\tr}$. Note that we have an exact sequence
		\begin{equation*}
			0	\to	\tilde{L}_\pi	\to	\tilde{A}_\pi^{\mathbf{m}}	\to	\calA_\pi^{{\mathbf{m}},\tr}	\to	0
		\end{equation*}
		and recall that $\tilde{L}_\pi$ is a finite free $R_q$-module of rank $N$. Let ${B}_{q,0}^{\mathbf{m}} = \{{e}_{0,k}^{\mathbf{m}}:1 \leq k \leq N\}$ be any $R_q$-basis for $\tilde{L}_\pi$. For $P$ and each $k > n(P)$, choose a lifting
		\begin{equation*}
			{e}_{P,k} = t_P^{-k} + c_{P,k} \in \calA_{\pi,P}^{\mathbf{m}},
		\end{equation*}
		where $\pr(c_{P,k}) = 0$. Then we define ${B}_{q,P}^{\mathbf{m}} = \{{e}_{P,k}^{\mathbf{m}}:k > \mu(P)\}$, where
		\begin{equation*}
			{e}_{P,k}^{\mathbf{m}} = \begin{cases}
					\pi^{k/m_P}{e}_{P,k}	&	\eta(P) = 0 \text{ or }\infty	\\
					\pi^{a(k)/m_P}{e}_{P,k}	&	\eta(P) = 1
				\end{cases}.
		\end{equation*}
		Our desired formal basis for $\tilde{V}_\pi^\mathbf{m}$ is:
		\begin{equation}\label{eq:globalbasis}
			{B}_q^{\mathbf{m}}	=	{B}_{q,0}^{\mathbf{m}} \sqcup \bigsqcup_P {B}_{q,P}^{\mathbf{m}}.
		\end{equation}
		
		\begin{proposition}[Global Column Estimate]\label{p:globalcolumn}
			Suppose that $\mathbf{m} \geq \mathbf{m}_\pi$. Then:
			\begin{enumerate}
				\item	If $\eta(P) = 0$ or $\infty$, then for all $k > 0$
					\begin{equation*}
						\tilde{\Theta}(e_{P,k}^\mathbf{m}) \in \pi^\frac{k(p-1)}{pm_P} \tilde{V}_{\pi,P}^\mathbf{m}.
					\end{equation*}
				\item	If $\eta(P) = 1$, then for all $k = p\ell+r > p-1$ with $0\leq r < p$
					\begin{equation*}
						\tilde{\Theta}(e_{P,k}^\mathbf{m}) \in \pi^\frac{\ell}{m_P} \tilde{V}_{\pi,P}^\mathbf{m}.
					\end{equation*}
			\end{enumerate}
		\end{proposition}
		\begin{proof}
			We consider each case separately:
			\begin{enumerate}
				\item	If $\eta(P) = 0$ or $\infty$ then we have
				\begin{equation*}
					\tilde{\Theta}({e}_{P,k}^{\mathbf{m}}) = \tilde{\Theta}(\pi^{k/m_P}t_P^{-k}) + \pi^{k/m_P}\tilde{\Theta}(c_k).
				\end{equation*}
				The first term lies in $\pi^{k(p-1)/pm_P} \calA_{\pi,P}^{{\mathbf{m}},\tr}$ by Proposition \ref{p:case1}. Moreover, we know that $\tilde{\Theta}$ restricts to an endomorphism of $\calA_{\pi,P}^{\mathbf{m}}$, so that the second term lies in $\pi^{k/m_P} \calA_{\pi,P}^{\mathbf{m}}$. The claim follows since $\pi^k\calA_\pi^\mathbf{m} \cap \tilde{A}_\pi^\dagger = \pi^k(\calA_\pi^\mathbf{m} \cap \tilde{A}_\pi^\dagger)$
				\item	 If $\eta(P) = 1$, then
			\begin{equation*}
				\tilde{\Theta}({e}_{P,k}^{\mathbf{m}}) = \tilde{\Theta}(\pi^{a(k)/m_P}t_P^{-k}) + \pi^{a(k)/m_P}\tilde{\Theta}(c_{P,k}).
			\end{equation*}
			By assumption, $\mathbf{m} \geq \mathbf{m}_\pi$ so that Proposition \ref{p:case2} holds. Thus, the first term lies in $\pi^{\ell e} \calA_\pi^\mathbf{m}$. The $\pi$-adic valuation of the second term is at least $a(k)/m_P \geq \ell/m_P$. The claim follows exactly as above.
			\end{enumerate}
		\end{proof}
		
		\begin{corollary}\label{t:good basis computes on Vdagger}
			For all $\mathbf{m} \geq \mathbf{m_\pi}$, the characteristic series $C(\Theta_q|V_\pi^\dagger,s)$ agrees with the Fredholm determinant $\det(I-s\tilde{\Theta}_q|B_q^{\mathbf{m}})$.
		\end{corollary}
		\begin{proof}
			First, observe that $\tilde{V}_\pi^\dagger$ is a union of Banach subspaces
			\begin{equation*}
				\tilde{V}_\pi^\dagger	=	\bigcup_\mathbf{m} \tilde{V}_\pi^{\mathbf{m}}.
			\end{equation*}
			Let $B^\mathbf{m}$ be an associated $K$-basis to $B_q^\mathbf{m}$, as in \S \ref{ss:iteration}. By the estimates in Proposition \ref{p:globalcolumn}, if $\mathbf{m} \geq \mathbf{m}_\pi$ then the matrix $\tilde{\Theta}|B^\mathbf{m}$ is tight. It follows that $\tilde{\Theta}_q|B_q^\mathbf{m}$ is tight, and so $\tilde{\Theta}_q$ acts completely continuously on $\tilde{V}_\pi^\mathbf{m}$. By Proposition \ref{p:basis}, the Fredholm series $\det(I-s\psi|\tilde{V}_\pi^\mathbf{m})$ is \emph{independent} of $\mathbf{m}$. The claim follows as in Example \ref{ex:nuclear}(\ref{i:union}).
		\end{proof}
		
		\begin{corollary}[Global Hodge Bound]\label{c:globalHodgeBound}
			The Newton polygon $\NP_{\pi_q}^{<v_\pi(p)}(\Theta_q|V_\pi^\dagger)$ lies on or above the convex polygon with slope set:
			\begin{equation*}
				\{	\underbrace{0,...,0}_r	\}	\sqcup	\bigsqcup_{P \in S}	\left\{	\frac{k(p-1)}{\delta_P}:1 \leq k  < v_\pi(p)\delta_P\right\}.
			\end{equation*}
		\end{corollary}
		\begin{proof}
			Fix a positive number $e \leq v_\pi(p)$, and consider the tuple $\mathbf{m}_e = (m_{e,P})_{P \in S_\eta}$ defined by
			\begin{equation*}
				m_{e,P} = \begin{cases}
				\frac{\delta_P}{p} & P \in S \\
				\frac{1}{pe} & P \not\in S \text{ and } \eta(P) \in \{0,\infty\} \\
				\frac{1}{e} &  \eta(P) = 1 
			\end{cases}.
			\end{equation*}
			Then $\mathbf{m}_e \geq \mathbf{m}_\pi$. Since the matrix of $\tilde{\Theta}$ with respect to ${B}^{\mathbf{m}_e}$ has coefficients in $R$, we see that $v_\pi(\tilde{\Theta}({e}_{0,k}^{\mathbf{m}_e})) \geq 0$ for all $1 \leq k \leq N$. Now let $B^{\mathbf{m}_e}$ be an associated $K$-basis to $B_q^{\mathbf{m}_e}$, as in \S \ref{ss:iteration}. By Proposition \ref{p:globalcolumn}, the column Hodge polygon $\cHP_\pi^{< e}(\tilde{\Theta}|B^{\mathbf{m}_e})$ lies on or above the convex polygon with slope set:
			\begin{equation*}
				\{	\underbrace{0,...,0}_{r v_p(q)}	\}	\sqcup	\bigsqcup_{P \in S}	\left\{	\frac{k(p-1)}{\delta_P}:1 \leq k  < e\delta_P\right\}^{\times{v_p(q)}}.
			\end{equation*}
			Upon taking the limit $e \to v_\pi(p)^+$ (or setting $e = v_\pi(p)$ if $R$ has characteristic $0$) the claim follows from Corollary \ref{t:good basis computes on Vdagger} and Lemma \ref{l:root}.
		\end{proof}
		
\section{Perturbing Operators and Main Results}

	In this final section we will study the interaction of the Newton polygon $\NP_{\pi_q}(\tilde{\Theta}_q|\tilde{V}_\pi^\dagger)$ with certain Hodge polygons attached to the action of $\tilde{\Theta}$ on the Banach spaces $\tilde{V}_\pi^\mathbf{m}$. For this purpose we fix a positive rational number $e \leq v_\pi(p)$, and let $\mathbf{m}_e$ be defined as in the proof of Corollary \ref{c:globalHodgeBound} (when $R$ has characteristic $0$, we may take $e = v_\pi(p)$).
	
	\subsection{Newton-Hodge Interaction and Perturbation Theory}
	\label{ss:perturbation}
		
		Let $I$ be a countable set. Let $\Psi$ be a tight matrix with entries in $R$ indexed by $I$, regarded as a completely continuous operator $\psi:b(I) \to b(I)$. We will now discuss the interaction between the Newton, Hodge, and column Hodge polygons of $\Psi$. For every $r > 0$, let
		\begin{equation*}
			I^{< r}(\Psi) = \{i \in I:v_\pi\psi(e_i) < r\}.
		\end{equation*}
		Let $\Psi^{< r}$ denote the finite diagonal submatrix of $\Psi$ indexed by $I^{< r} (\Psi)$. We are primarily interested in the following strong type of Newton-Hodge interaction:
		
		\begin{lemma}\label{l:NHinteraction}
			The following are equivalent:
			\begin{enumerate}
				\item	$\NP_\pi^{< r}(\Psi)$ and $\cHP_\pi^{< r}(\Psi)$ have the same terminal point.
				\item	$\NP_\pi(\Psi^{< r})$ and $\cHP_\pi^{< r}(\Psi)$ have the same terminal point.
			\end{enumerate}
		\end{lemma}
		\begin{proof}
			Choose an ordering on $I$ so that the column slopes $v_\pi \psi(e_i)$ are increasing. Label the elements $I^{<r} = \{i_1,...,i_n\}$ in increasing order, and let $(n,m)$ denote the terminal point of $\cHP_\pi^{< r}(\Psi)$ so that
			\begin{equation*}
				m	=	v_\pi \Psi(e_{i_1})+\cdots + v_\pi\Psi(e_{i_n}).
			\end{equation*}
			Let $J = \{j_1,...,j_n\}$ be any subset of $I$ of cardinality $n$, with elements labelled in increasing order. Then for every $1 \leq k \leq n$ we have
			\begin{equation*}
				v_\pi \Psi(e_{j_k}) \geq v_\pi \Psi(e_{i_k}).
			\end{equation*}
			If $J \neq I^{< r}(\Psi)$, then the inequality is \emph{strict} for $k = n$, since $e_{j_n} \notin I^{< r}(\Psi)$. Since $\NP_\pi(\Psi_J) \succeq \cHP_\pi(\Psi_J)$, we see that every $J \times J$ minor of $\Psi$ has $\pi$-adic valuation $\geq v_\pi \det(\Psi^{< r})$, with strict inequality when $J \neq I^{< r}(\Psi)$. Thus $\NP_\pi^{< r}(\Psi)$ and $\cHP_\pi^{< r}(\Psi)$ have the same terminal point if and only if $v_\pi \det(\Psi^{< r}) = m$, i.e. if and only if $\NP_\pi(\Psi^{< r})$ and $\cHP_\pi^{< r}(\Psi)$ have the same terminal point
		\end{proof}
		
		\begin{lemma}\label{l:HPequal}
			Suppose that $\NP_\pi^{< r}(\Psi)$ and $\cHP_\pi^{< r}(\Psi)$ have the same terminal point. Then the polygons
			\begin{equation*}
				\HP_\pi(\Psi^{< r}),\ \cHP_\pi(\Psi^{< r}),\ \HP_\pi^{< r}(\Psi),\ \cHP_\pi^{< r}(\Psi)
			\end{equation*}
			are all equal.		
		\end{lemma}
		\begin{proof}
			Observe that each of the listed polygons lies on or below $\HP_\pi(\Psi^{< r})$ and on or above $\cHP_\pi^{< r}(\Psi)$. By Lemma \ref{l:NHinteraction}, all of these polygons share the same terminal point. Since each column slope of $\Psi^{< r}$ is greater than or equal to the corresponsing column slope of $\Psi$, we must have that $\cHP_\pi(\Psi^{< r}) = \cHP_\pi(\Psi)$. It suffices then to consider the case that $\Psi = \Psi^{<r}$ is a finite matrix whose $\pi$-adic column slopes are $< r$. But this is immediate from Lemma \ref{l:hodgeSlopes}.
		\end{proof}

		\begin{definition}
			Let $\Psi' = \Psi + \varepsilon$, where $\varepsilon$ is a matrix with entries in $R$ indexed by $I$. We say that $\Psi'$ is a $\pi$-adic \emph{$r$-perturbation} of $\Psi$ if:
			\begin{enumerate}
				\item	$v_\pi \varepsilon(e_i) > v_\pi \Psi(e_i)$ for all $i \in I^{< r}(\Psi)$.
				\item	$v_\pi \varepsilon(e_i) \geq r$ for all $i \notin I^{< r}(\Psi)$.
			\end{enumerate}
		\end{definition}
		
		If $\Psi'$ is a tight $\pi$-adic $r$-perturbation of $\Psi$, then the truncated column Hodge polygons $\cHP^{< r}_\pi (\Psi)$ and $\cHP^{< r}_\pi (\Psi')$ of Lemma \ref{l:hodgeSlopes} necessarily agree. The stronger perturbation condition guarantees the following interaction between the Newton polygons of both matrices:
		
		\begin{lemma}[Perturbation Lemma]\label{l:perturb}
			Suppose that $\Psi' = \Psi+\varepsilon$ is a tight $\pi$-adic $r$-perturbation of $\Psi$. If $\NP_\pi^{< r} (\Psi)$ has the same terminal point as $\cHP_\pi^{< r} (\Psi) = \cHP_\pi^{<r} (\Psi)$, then the same is true of $\NP_\pi^{< r} (\Psi')$.
		\end{lemma}
		\begin{proof}
			By Lemmas \ref{l:NHinteraction} and \ref{l:HPequal}, we immediately reduce to the case that $\Psi$ and $\Psi'$ are finite matrices whose column slopes are all $< r$. We must show that $v_\pi \det(\Psi) = v_\pi \det(\Psi')$. Choose an ordering of $I$ so that the sequence $v_\pi \Psi(e_i)$ is increasing. Given an index $(i_1,...,i_k) \in \wedge^n I$, we will abbreviate $e_{i_1,...,i_k} = e_{i_1} \wedge \cdots \wedge e_{i_k}$. Let us write
			\begin{equation*}
				\wedge^k \Psi' = \wedge^k \Psi + \varepsilon_k
			\end{equation*}
			for all $k \geq 0$. A straightforward computation shows that
			\begin{equation*}
				\varepsilon_{k+1}(e_{i_1,...,i_{k+1}}) = \wedge^k \Psi(e_{i_1,...,i_k}) \wedge \varepsilon(e_{i_{k+1}}) + \varepsilon_k(e_{i_1,...,i_k})\wedge \Psi'(e_{i_{k+1}}).
			\end{equation*}
			Since $\Psi'$ is an $r$-perturbation of $\Psi$, by induction on $k$ we see that
			\begin{equation*}
				v_\pi(\varepsilon_{n})	>	v_\pi \Psi(e_1) + \cdots + v_\pi \Psi(e_n).
			\end{equation*}
			By assumption, the latter is exactly $\det(\Psi)$. It follows that
			\begin{equation*}
				v_\pi \det(\Psi') = v_\pi (\det(\Psi) + \varepsilon_n) = v_\pi \det(\Psi).
			\end{equation*}
		\end{proof}
		
	\subsection{Local-to-global extensions}\label{ss:localtoglobal}
	
		Let $P \in S$. Let $G_P$ denote the absolute Galois group of $F_P$, and let $\rho_P$ denote the restriction of $\rho$ to $G_P$. By a theorem of Katz-Gabber \cite{Katz2}, $\rho_P$ extends in a canonical way to a continuous character
		\begin{equation*}
			\rho_P^\ext:\pi_1(\A_{\F_q}^1)	\to	R^\times.
		\end{equation*}
		Let us describe $\rho_P^\ext$ explicitly in terms of $\sigma$-modules. We define a flat lifting $(A_P,\sigma)$ of $\A_{\F_q}^1$ over $\Z_p$, where $A_P = \Z_q[t_P^{-1}]$ and $\sigma(t_P^{-1}) = t_P^{-p}$. Let $A_{\pi,P} = R \otimes_{\Z_p} A_P$. Since $\tilde{E}$ is a $\bm\delta$-Frobenius structure for $(M,\phi)$, we may regard $\tilde{E}_P$ as an element of $A_{\pi,P}^{\delta_P}$. Then $\rho_P^\ext$ corresponds to the unit-root $\sigma$-module
		\begin{equation*}
			(M_P^\ext,\phi_P^\ext)	=	(A_{\pi,P}^\dagger,\tilde{E}_P \circ \sigma).
		\end{equation*}
		
		Since $\rho$ is $\bm\delta$-overconvergent, the Dwork trace formula (i.e. the Monsky trace formula over $\G_m$) guarantees that $L(\rho_P^\ext,s)$ is analytic in the disk $v_\pi(s) > - v_\pi(q)$. To state the trace formula, we define as before
		\begin{equation*}
			U_p	=	\frac{1}{p}	\sigma^{-1} \circ \Tr:A_P^\dagger \to A_P^\dagger.
		\end{equation*}
		By weak base change, $U_p$ induces a $p$-Dwork operator on $A_{\pi,P}^\dagger$. Consider the $p$-Dwork operator $\tilde{\Theta} = U_p \circ \tilde{E}_P$. Let $V_{\pi,P}^\dagger = K \otimes_R A_{\pi,P}^\dagger$. Then the action of $\tilde{\Theta}_q$ on $V_{\pi,P}^\dagger$ is nuclear, and the Dwork trace formula asserts that
		\begin{equation*}
			(1-\alpha s)L(\rho_P^\ext,s) = \frac{C(\tilde{\Theta}_q|V_{\pi,P}^\dagger,s)}{C(\tilde{\Theta}_q|V_{\pi,P}^\dagger,qs)},
		\end{equation*}
		where $\alpha \in 1+\pi R_q$ denotes constant term of $\tilde{E}_P$. Equivalently,
		\begin{equation*}
			C(\tilde{\Theta}_q|V_{\pi,P}^\dagger,s) = \prod_{j = 0}^\infty (1-\alpha 	q^js)L(\rho_P^\ext,q^js).
		\end{equation*}
		
		\begin{lemma}\label{l:A1factorization}
			There is a factorization
			\begin{equation*}
				C(\tilde{\Theta}_q|V_{\pi,P}^\dagger,s) = (1-\alpha s)\cdot C(\tilde{\Theta}_q|\calV_{\pi,P}^{\dagger,\tr},s).
			\end{equation*}
		\end{lemma}
		\begin{proof}
			We have a decomposition $A_{\pi,P}^\dagger	=	R_q \oplus \calA_{\pi,P}^{\dagger,\tr}$. We have already seen that the rightmost factor is $\tilde{\Theta}_q$-invariant and that $\tilde{\Theta}_q$ acts as a nuclear operator on this space. The desired factorization is obtained as in Example \ref{ex:nuclear}(\ref{i:quotient}).
		\end{proof}
		
		\begin{definition}
			Let $P \in S$. The $\pi_q$-adic \emph{local Newton polygon} of $\rho_P$ is
			\begin{equation*}
				\NP_{\pi_q} (\rho_P) = \NP_{\pi_q} (\tilde{\Theta}_q|\calV_{\pi,P}^{\dagger,\tr}).
			\end{equation*}
		\end{definition}
		
		By the Dwork trace formula, we see that $\NP_{\pi_q}^{< e}(\rho_P) = \NP_{\pi_q}^{< e}(\rho_P^\ext)$, and in particular this polygon does not depend on our local lifting $(\calA_P,\sigma)$. We now turn our attention to the Hodge polygons of local-to-global extensions. First, for each $P \in S$ we define the \emph{$\delta_P$-Hodge polygon} $\HP(\delta_P)$ to be the convex polygon with slope set
		\begin{equation*}
			\left\{	\frac{p-1}{\delta_P},\frac{2(p-1)}{\delta_P},...	\right\}.
		\end{equation*}
		Recall that since $P \in S$, the growth condition $m_{e,P} = \delta_P/p$ does not depend on $e$. By Proposition \ref{p:case1}, we know that
		\begin{equation}\label{eq:deltaHodge}
			\HP_\pi(\tilde{\Theta}|\calV_{\pi,P}^{\mathbf{m}_e,\tr}) \succeq \HP(\delta_P)^{\times v_p(q)}.
		\end{equation}
		
		\begin{definition}
			We say that $\rho_P$ is $\pi$-adically \emph{$\delta_P$-Hodge} if (\ref{eq:deltaHodge}) is an equality. In this case, we define the $\pi_q$-adic \emph{local Hodge polygon} of $(M,\phi)$ at $P$ to be
			\begin{equation*}
				\HP_{\pi_q}(\rho_P) = \HP(\delta_P). 
			\end{equation*}
			We say that $\rho$ is $\pi$-adically \emph{$\bm\delta$-Hodge} if $\rho_P$ is $\pi$-adically $\delta_P$-Hodge for all $P \in S$.
		\end{definition}
		
		\begin{remark}
			In general, the local Hodge polygon $\HP_{\pi_q}(\rho_P)$ \emph{need not} agree with the Hodge polygon $\HP_{\pi_q}(\tilde{\Theta}_q|\calV_{\pi,P}^{\mathbf{m}_e,\tr})$. However, by Lemma \ref{l:root} we always have
			\begin{equation*}
				\NP_{\pi_q} (\rho_P) \succeq \HP_{\pi_q}(\rho_P).
			\end{equation*}
		\end{remark}
		
		\begin{proposition}\label{p:deltaHodgeTouching}
			Let $r_i$ be a sequence of positive rational numbers with $r_i \to \infty$. Suppose that for each $i$, the polygons $\NP_{\pi_q}^{< r_i}(\rho_P)$ and $\HP^{< r_i}(\delta_P)$ have the same terminal point. Then $\rho_P$ is $\pi$-adically $\delta_P$-Hodge.
		\end{proposition}
		\begin{proof}
			By Lemma \ref{l:root} we see that
			\begin{equation*}
				\NP_\pi(\tilde{\Theta}|\calV_{\pi,P}^{\dagger,\tr}) \succeq \HP_\pi(\tilde{\Theta}|\calV_{\pi,P}^{\mathbf{m}_e,\tr}) \succeq \cHP_\pi(\tilde{\Theta}|B_P^{\mathbf{m}_e}) \succeq \HP(\delta_P)^{\times v_p(q)}.
			\end{equation*}
			By the estimates in \S\ref{ss:case1}, we know that every column slope of $\tilde{\Theta}|B_P^{\mathbf{m}_e}$ is greater than or equal to the corresponding column slope of $\HP(\delta)^{\times v_p(q)}$. By assumption, the $< r_i$-truncation of all four polygons have the same terminal point for all $i$. Since $r_i \to \infty$, it follows from Lemma \ref{l:HPequal} that
			\begin{equation*}
				\HP_\pi (\tilde{\Theta}|V_{\pi,P}^{\mathbf{m}_e,\tr}) = \cHP_{\pi_q} (\tilde{\Theta}|B_P^{\mathbf{m}_e}) = \HP(\delta_P)^{\times v_p(q)}.
			\end{equation*}
		\end{proof}
		
		\begin{theorem}\label{t:finite are delta-good}
			Let $\rho:\pi_1(X) \to R^\times$ be a finite character of order $p^n$. Let $\pi \in \Z_p[\rho]$ be a uniformizer. For each $P \in S$, let $d_P$ denote the Swan conductor of $\rho$ at $P$, and let $\delta_P = d_P/p^{n-1}$. Then:
			\begin{enumerate}
				\item	\label{i:periodicity}For all $P \in S$ and all $n > 0$, the polygons $\NP_{\pi_q} (\rho_P)$ and $\HP(\delta_P)$ agree on the interval $[nd_P-1,nd_P]$.
				\item	\label{i:dhodge}$\rho$ is $\pi$-adically $\bm\delta$-Hodge.
			\end{enumerate}
		\end{theorem}
		\begin{proof}
			Since $L(\rho_P^\ext,s)$ is a polynomial of degree $d_P-1$, we see that $\NP_{\pi_q}^{< e} (\rho_P^\ext)$ and $\HP^{< e}(\delta)$ share the same terminal point $(d_P-1,(d_P-1)/\delta_P)$. Claim \ref{i:periodicity} follows immediately from the Dwork trace formula, and claim \ref{i:dhodge} follows from Proposition \ref{p:deltaHodgeTouching}.
		\end{proof}
		
		In the follow-up papers \cite{Kramer-Miller-Upton2}, \cite{Upton2} we will verify that each of the $\bm\delta$-overconvergent characters of Example \ref{ex:deltaOverconvergent} are also $\bm\delta$-Hodge.
		
	\subsection{Perturbations Coming from Local-to-Global Extensions}
	\label{ss:global touching}
	
		We will now define a perturbation of the matrix $\Psi = \tilde{\Theta}|{B}^{\mathbf{m}_e}$ using local-to-global extensions. Using the partition (\ref{eq:globalbasis}), we may regard $\Psi = (\Psi_{\alpha,\beta})$ as a block matrix, where $\alpha,\beta$ are either $0$ or one of the $P$. Let us define a new block matrix $\Psi' = (\Psi'_{\alpha,\beta})$ as follows: For any $\alpha$, we let $\Psi'_{\alpha,0} = \Psi_{\alpha,0}$. If $P \in S$, then we let $\Psi'_{P,P} = \tilde{\Theta}|B_P^{\mathbf{m}_e}$. We define all other blocks to be zero.
		
		\newcommand{\Pf}{\mathbf{P}}
		\newcommand{\UP}[2]{\makebox[0pt]{\smash{\raisebox{1.5em}{$\phantom{#2}#1$}}}#2}
		\newcommand{\LF}[1]{\makebox[0pt]{$#1$\hspace{5em}}}
		
		\begin{figure}
			\begin{equation*}
				\renewcommand{\arraystretch}{1.3}
				\left[\begin{array}{c@{}ccc|ccc|ccc}
					\LF{0}&&	\UP{0}{\Psi_{0,0}}	&	&	0	&	\UP{P \in S}{\cdots}	&	0	&	0	&	\UP{P \notin S}{\cdots}	&	0	\\\hline
					&&	\Psi_{0,P}	&	&	\tilde{\Theta}|B_P^{\mathbf{m}_e}	&	\cdots	&	0	&	0	&	\cdots	&	0	\\
					\LF{P \in S}&&	\vdots	&	&	\vdots	&	\ddots	&	\vdots	&	\vdots	&	\ddots	&	\vdots	\\
					&&	\Psi_{0,P}	&	&	0	&	\cdots	&	\tilde{\Theta}|B_P^{\mathbf{m}_e}	&	0	&	\cdots	&	0	\\\hline
					&&	\Psi_{0,P}	&	&	0	&	\cdots	&	0	&	0	&	\cdots	&	0	\\
					\LF{P \notin S}&&	\vdots	&	&	\vdots	&	\ddots	&	\vdots	&	\vdots	&	\ddots	&	\vdots	\\
					&&	\Psi_{0,P}	&	&	0	&	\cdots	&	0	&	0	&	\cdots	&	0	\\
					
				\end{array}\right]		
			\end{equation*}
			\caption{The shape of the block matrix $\Psi'$, with block indices labelled.}
		\end{figure}
		
		\begin{lemma}\label{l:globalperturbation}
			If $\rho$ is $\pi$-adically $\bm\delta$-Hodge, then $\Psi'$ is an $e$-perturbation of $\Psi$.
		\end{lemma}
		\begin{proof}
			First, recall from the global Hodge bound that the $\pi$-adic valuation of $\tilde{\Theta}(e_{\alpha,k}^{\mathbf{m}_e})$ is $\geq e$ if either: $\alpha = P \notin S$, or $\alpha = P \in S$ and $k \geq e \delta_P$. Thus we need only consider the cases $\alpha = 0$ and $\alpha = P \in S$ with $k < e \delta_P$. Let $\Psi' = \Psi+\varepsilon$, and regard $\varepsilon = (\varepsilon_{\alpha,\beta})$ as a block matrix. Then by construction we see that $\varepsilon_{\alpha,0} = 0$ regardless of $\alpha$, so for these columns $\Psi'$ satisfies the perturbation condition.
			
			It remains to consider the case $\alpha = P \in S$ and $k < e \delta_P$. Write ${e}_{P,k} = t_P^{-k}+c_{P,k}$, so that
			\begin{equation*}
				\tilde{\Theta}({e}_{P,k}^{\mathbf{m}_e}) = \tilde{\Theta}(\pi^{kp/\delta_P}t^{-k}) + \pi^{kp/\delta_P}\tilde{\Theta}(c_k).
			\end{equation*}
			Regarding $\varepsilon$ as an endomorphism of $\tilde{A}_\pi^{\mathbf{m}_e}$, we see that
			\begin{equation*}
				\varepsilon({e}_{P,k}^{\mathbf{m}_e}) = -\pi^{kp/\delta_P}\tilde{\Theta}(c_k).
			\end{equation*}
			It follows that
			\begin{equation*}
				v_\pi(\varepsilon({e}_{P,k}^{\mathbf{m}_e})) \geq \frac{kp}{\delta_P}.
			\end{equation*}
			On the other hand, the assumption that $\rho$ is $\pi$-adically $\bm\delta$-Hodge implies that
			\begin{equation*}
				v_\pi(\tilde{\Theta}({e}_{P,k}^{\mathbf{m}_e})) = \frac{k(p-1)}{\delta_P}.
			\end{equation*}
		\end{proof}
		
	\subsection{Main results}
	
		We are now ready to state and prove the general form of Theorem \ref{intro theorem: local touching of polygons givesglobal touching}. Let $r \leq v_\pi(p)$, so that $L(\phi,s)$ is analytic in the region $v_{\pi_q}(s) > -r$ by the Monsky trace formula. We define the $r$-truncated \emph{$\pi_q$-adic Newton polygon} of $(M,\phi)$ to be
		\begin{equation*}
			\NP_{\pi_q}^{<r} (\phi)=	\NP_{\pi_q}^{<r} (\Theta_q|V_\pi^\dagger).
		\end{equation*}
		
		\begin{lemma}\label{l:slope0}
			If $\overline{X}$ is ordinary, then $\NP_{\pi_q}^{< r} (\phi)$ has $N$ segments of slope $0$.
		\end{lemma}
		\begin{proof}
			Observe that $(M,\phi)$ is congruent to the trivial $\sigma$-module $(A_\pi^\infty,\sigma)$ modulo $\pi$. In particular, if we write $P(U,s)/(1-qs)$ for the zeta function of $U$, then
			\begin{equation*}
				L(\phi,s)	\equiv	P(U,s)	\pmod{\pi}.
			\end{equation*}
			The result follows from the Deuring-Shafarevich formula.
		\end{proof}
		
		Assume now that $\rho$ is $\pi$-adically $\bm\delta$-Hodge. We define the $r$-truncated \emph{$\pi_q$-adic Hodge polygon} of $(M,\phi)$ to be
		\begin{equation*}
			\HP_{\pi_q}^{<r} (\phi)	=	\NP \{\underbrace{0,\dots,0}_{N}\}\sqcup \bigsqcup_{P \in S} \HP_{\pi_q}^{< r} (\rho_P^\ext).
		\end{equation*}
		
		\begin{theorem}\label{t:main}
			Suppose that $\overline{X}$ is ordinary and that $\rho$ is $\pi$-adically $\bm\delta$-Hodge. Then $\NP_{\pi_q}^{<r} (\rho_P^\ext)$ and $\HP_{\pi_q}^{<r} (\rho_P^\ext)$ have the same terminal point for each $P \in S$ if and only if $\NP_{\pi_q}^{<r} (\phi)$ and $\HP_{\pi_q}^{<r} (\phi)$ have the same terminal point.
		\end{theorem}
		\begin{proof}
			Since $r \leq v_\pi(p)$, we may enlarge $e$ as necessary and assume that $r \leq e$. By Lemma \ref{l:globalperturbation} the matrix $\Psi'$ of \ref{ss:global touching} is a $\pi$-adic $r$-perturbation of $\Psi$. Since $\overline{X}$ is assumed to be ordinary, we know that $\NP_\pi^{< r} (\Psi) = \NP_{\pi_q}^{< r} (\phi)^{\times{v_p(q)}}$ must have $N v_p(q)$ segments of slope $0$. By the perturbation lemma, the same is true of $\Psi'$. Thus
			\begin{equation*}
				\NP_\pi^{<r} (\Psi') = \NP \{\underbrace{0,\dots,0}_{N v_p(q)}\} \sqcup \bigsqcup_{P \in S} \NP_{\pi_q}^{<r} (\rho_P^\ext)^{\times{v_p(q)}}.
			\end{equation*}
			and
			\begin{equation*}
				\cHP_\pi^{< r} (\Psi') = \NP\{\underbrace{0,\dots,0}_{N v_p(q)}\}\sqcup \bigsqcup_{P \in S} \HP_{\pi_q}^{< r} (\rho_P^\ext)^{\times v_p(q)}.
			\end{equation*}
			The claim is immediate from the perturbation lemma.
		\end{proof}
		
		\begin{corollary}\label{c:delta-hodge result}
			Suppose that $\overline{X}$ is ordinary and that $\rho:\pi_1(X) \to R^\times$ is $\pi$-adically $\bm\delta$-Hodge. Let $r \leq v_\pi(p)$. If $\NP_{\pi_q}^{<r} (\rho_P^\ext)$ and $\HP_{\pi_q}^{<r} (\rho_P^\ext)$ have the same terminal point for each $P \in S$, then $\NP_{\pi_q}^{<r} (\rho)$ and $\HP_{\pi_q}^{<r} (\rho)$ have the same terminal point.
		\end{corollary}
		\begin{proof}
			Let $(M,\phi)$ be the unit-root $\sigma$-module over $A_\pi^\infty$ corresponding to $\rho_U$. From the equality (\ref{eq:ArtinLFunction}), the Newton polygon $\NP_{\pi_q}^{< r} (\phi)$ is the concatenation of $\NP_{\pi_q}^{< r}(\rho)$ with $r_0+r_1+r_\infty-|S|$ segments of slope $0$. By definition, the Hodge polygon $\HP_{\pi_q}^{< r}(\phi)$ is a concatenation of $\HP_{\pi_q}^{< r}(\rho)$ with $r_0+r_1+r_\infty-|S|$ segments of slope $0$. The result follows by canceling out these extra slope-$0$ segments.
		\end{proof}
		
		\begin{corollary} \label{c:theorem local touching holds}
			Theorem \ref{intro theorem: local touching of polygons givesglobal touching} holds.
		\end{corollary}
		\begin{proof}
			By Theorem \ref{t:finite are delta-good} we know that any finite character of order $p^n$
			is $\bm\delta$-Hodge, where $\delta_P$ is the Swan conductor at $P$ divided by $p^{n-1}$.
			The result is then an immediate consequence of Corollary \ref{c:delta-hodge result}.
		\end{proof}
	
		\begin{corollary} \label{c: theorem np=hp holds}
			Theorem \ref{intro theorem: HP is NP with congruence conditions} holds for finite characters of order $p$.
		\end{corollary}
		\begin{proof}
			By Remark \ref{r: only if remark} we only need to show the ``if'' direction. By our assumption on the order
			of $\rho$ we have $d_P \in \Z$ and $p\equiv 1 \mod d_P$. By \cite{Robba-lower_bounds_NP} (see also \cite[Remark 4.1]{Blache-Ferard-Newton_stratum})
			we know $\NP_q(\rho_P^\ext)=\HP_q(\rho_P^\ext)$ for each $P\in S$. The result follows from
			Theorem \ref{t:finite are delta-good}.
		\end{proof}

	\section{Glossary of notation}
	The spaces studied in this paper are numerous and intricate. For this reason, 
	we thought the reader would benefit from a notation glossary.
	\subsection{Basic definitions}
	\begin{enumerate}
		\item[$R$] A complete discrete valuation ring with maximal ideal $\frakm$ and residue field $\mathbb{F}_p$.
		\item[$\pi$] A non-zero topologically nilpotent element of $R$.
		\item[$K$] The fraction field of $R$, i.e. $K=R[\frac{1}{\pi}]$.
		\item[$q$] A power of $p$.
		\item[$R_q$] The unramified extension of $R$ whose residue field is $\mathbb{F}_q$. 
		\item[$K_q$] The fraction field of $R_q$, i.e. $K_q=R_q[\frac{1}{\pi}]$.
		\item[$\pi_q$] The $v_p(q)$-th power of $\pi$: i.e. $\pi_q=\pi^{v_p(q)}$.
		\item[$X$] A smooth affine curve over $\mathbb{F}_q$.
		\item[$\overline{X}$]  The smooth compactification of $X$.
		\item[$S$] The points of $\overline{X}$ at infinity, i.e. $S=\overline{X}\backslash X$.
		\item[$\eta$] A tame morphism $\overline{X} \to \mathbb{P}_{\mathbb{F}_q}^1$
		that is only ramified over $\{0,1,\infty\}$. See \S \ref{ss:mapping} for more details.
		\item[$S_\eta$] The branch points of $\eta$, i.e. $S_\eta = \eta^{-1}(\{0,1,\infty\})$.
		\item[$\mu$] A function from $S_\eta$ to $\Z$. More precisely, if $\eta(P)\in\{0,\infty\}$ we 
		have $\mu(P)=0$ and if $\eta(P)=1$ we have $\mu(P)=p-1$.
		\item[$A$] A smooth $\Z_q$-algebra whose special fiber is the coordinate ring of
		$\overline{X}-S_\eta$.
	\end{enumerate}
	
	\subsection{The space \texorpdfstring{$A$}{A}, its modifications, and some associated notation}
		In this article, we frequently work with modifications to the space of functions $A$.
		These modified spaces are notated using combinations of superscripts, subscripts, and font changes. These modifiers are occasionally applied to spaces other than $A$. However, 
		for expository purposes, it seems prudent to use $A$ as our `running example'. We will now
		explain these notation modifiers, as well as explain some other notation for objects associated to $A$.
		
		\begin{enumerate}
			\item[$A_\pi$] The ring $R \otimes_{\Z_p} A$.
			\item[$A_\pi^\infty$] The $\pi$-adic completion of $A_\pi$.
			\item[$A_\pi^{m}$] The elements of $A_{\pi}^\infty$ that
			`overconverge' $\pi$-adically with a radius of $m$. See \S \ref{ss:growth} for more details.
			\item[$A_\pi^\dagger$] The elements of $A_{\pi}^\infty$ that `overconverge' $\pi$-adically. In particular, it is the union of all $A_\pi^m$.
			\item[$V_\pi^\dagger$] The $K$-vector space $K \otimes_R A_\pi^\dagger$. More
			generally, when $A$ is replaced with $V$, this signifies tensoring with $K$.
			\item[$\tilde{A}_\pi^{\dagger}$] The semi-local `twist' of $A_\pi^\dagger$.
			\item[$\mathcal{A}_{\pi,P}$] The completion of $A_\pi$ along the ideal defining $P$.
			We use the superscripts $m$, $\dagger$, and $\infty$ for the appropriate $\pi$-adic
			growth conditions. 
			\item[$\mathbf{m}$] A tuple $(m_P)_{P \in S_\eta}$ of rational numbers indexed by $S_\eta$.
			\item[$\mathcal{A}_{\pi,P}^{\mathbf{m}}$] This is
			a subspace of $\mathcal{A}_{\pi,P}^\dagger$ with very precise growth conditions
			depending on $\eta(P)$ and $\mathbf{m}$. See \S \ref{ss:case1} and \S \ref{ss:case2} for the precise definitions. 
			\item[$\mathcal{A}_{\pi,P}^\tr$] The subspace of elements in $\mathcal{A}_P$
			of `truncated' power series. More precisely, elements
			that can be written as $\sum\limits_{k=\mu(P)}^\infty a_k t_P^{-k}$.
			This will frequently be combined with another `growth' superscript, such as
			$\dagger$, $\infty$, or $\mathbf{m}$.
			\item[$\mathcal{A}_\pi$] The product $\prod\limits_{P \in S_\eta} \mathcal{A}_{\pi,P}$. 
			We use the superscripts $\mathbf{m}$, $\dagger$, and $\infty$ for the appropriate $\pi$-adic
			growth conditions. We also use the $\tr$ superscript to denote the `truncated' power series in each summand. 
		
		\end{enumerate}
		
\bibliographystyle{plainnat}
\bibliography{Stable}

\begin{thebibliography}{28}
\providecommand{\natexlab}[1]{#1}
\providecommand{\url}[1]{\texttt{#1}}
\expandafter\ifx\csname urlstyle\endcsname\relax
  \providecommand{\doi}[1]{doi: #1}\else
  \providecommand{\doi}{doi: \begingroup \urlstyle{rm}\Url}\fi

\bibitem[Adolphson and Sperber(1989)]{Adolphson-Sperber}
Alan Adolphson and Steven Sperber.
\newblock Exponential sums and {N}ewton polyhedra: cohomology and estimates.
\newblock \emph{Ann. of Math. (2)}, 130\penalty0 (2):\penalty0 367--406, 1989.
\newblock ISSN 0003-486X.
\newblock \doi{10.2307/1971424}.
\newblock URL \url{https://doi.org/10.2307/1971424}.

\bibitem[Blache and F\'{e}rard(2007)]{Blache-Ferard-Newton_stratum}
R\'{e}gis Blache and \'{E}ric F\'{e}rard.
\newblock Newton stratification for polynomials: the open stratum.
\newblock \emph{J. Number Theory}, 123\penalty0 (2):\penalty0 456--472, 2007.
\newblock ISSN 0022-314X.
\newblock \doi{10.1016/j.jnt.2006.06.009}.
\newblock URL \url{https://doi.org/10.1016/j.jnt.2006.06.009}.

\bibitem[Booher and Pries(2020)]{Booher-Pries}
Jeremy Booher and Rachel Pries.
\newblock Realizing {A}rtin-{S}chreier covers of curves with minimal {N}ewton
  polygons in positive characteristic.
\newblock \emph{J. Number Theory}, 214:\penalty0 240--250, 2020.
\newblock ISSN 0022-314X.
\newblock \doi{10.1016/j.jnt.2020.04.010}.
\newblock URL \url{https://doi.org/10.1016/j.jnt.2020.04.010}.

\bibitem[Crew(1987)]{Crew}
Richard Crew.
\newblock {$L$}-functions of $p$-adic characters and geometric {I}wasawa
  theory.
\newblock \emph{Inventiones Mathematicae}, 88:\penalty0 395--404, 1987.
\newblock URL \url{http://eudml.org/doc/143459}.

\bibitem[Davis et~al.(2012)Davis, Langer, and Zink]{Davis-Langer-Zink}
Christopher Davis, Andreas Langer, and Thomas Zink.
\newblock Overconvergent {W}itt vectors.
\newblock \emph{J. Reine Angew. Math.}, 668:\penalty0 1--34, 2012.
\newblock ISSN 0075-4102.
\newblock \doi{10.1515/CRELLE.2011.141}.
\newblock URL \url{https://doi.org/10.1515/CRELLE.2011.141}.

\bibitem[Davis et~al.(2016)Davis, Wan, and Xiao]{Davis}
Christopher Davis, Daqing Wan, and Liang Xiao.
\newblock Newton slopes for {Artin}-{Schreier}-{Witt} towers.
\newblock \emph{Math. Ann.}, 364\penalty0 (3-4):\penalty0 1451--1468, 2016.

\bibitem[Elkik(1973)]{Elkik}
Ren\'ee Elkik.
\newblock Solutions d'\'equations \`a coefficients dans un anneau
  {H}ens\'elien.
\newblock \emph{Annales Scientifiques de l'\'{E}cole Normale Sup\'{e}rieure},
  6\penalty0 (4):\penalty0 553--603, 1973.

\bibitem[Fulton(1969{\natexlab{a}})]{Fulton}
William Fulton.
\newblock A note on weakly complete algebras.
\newblock \emph{Bulletin of the American Mathematical Society}, 75\penalty0
  (3):\penalty0 591--593, 1969{\natexlab{a}}.

\bibitem[Fulton(1969{\natexlab{b}})]{FultonHurwitz}
William Fulton.
\newblock Hurwitz schemes and irreducibility of moduli of algebraic curves.
\newblock \emph{Annals of Mathematics}, 90\penalty0 (3):\penalty0 542--575,
  1969{\natexlab{b}}.

\bibitem[Jean-Pierre(1962)]{Serre}
Serre Jean-Pierre.
\newblock Endomorphismes compl\`etement continus des espaces de {B}anach
  $p$-adiques.
\newblock \emph{Publ. Math. I.H.E.S}, 12:\penalty0 69--85, 1962.

\bibitem[Katz(1973)]{Katz}
Nicholas~M. Katz.
\newblock $p$-adic properties of modular schemes and modular forms.
\newblock In Willem Kuijk and Jean-Pierre Serre, editors, \emph{Modular
  Functions in One Variable {III}}, volume 350. Springer, Berlin, Heidelberg,
  1973.

\bibitem[Katz(1986)]{Katz2}
Nicholas~M. Katz.
\newblock Local-to-global extensions of representations of fundamental groups.
\newblock \emph{Ann. Inst. Fourier (Grenoble)}, 36\penalty0 (4):\penalty0
  69--106, 1986.
\newblock ISSN 0373-0956.
\newblock URL \url{http://www.numdam.org/item?id=AIF_1986__36_4_69_0}.

\bibitem[Kedlaya et~al.(2020)Kedlaya, Litt, and Witaszek]{Kedlaya}
Kiran~S. Kedlaya, Daniel Litt, and Jakub Witaszek.
\newblock Tamely ramified morphisms of curves and {B}elyi's theorem in positive
  characteristic, 2020.

\bibitem[Kosters and Zhu(2018)]{KZ}
Michiel Kosters and Hui~June Zhu.
\newblock On slopes of {$L$}-functions of $\mathbb{Z}_p$-covers over the
  projective line.
\newblock \emph{Journal of Number Theory}, 187:\penalty0 430 -- 452, 2018.

\bibitem[Kramer-Miller()]{Kramer-Miller2}
Joe Kramer-Miller.
\newblock The monodromy of unit-root {$F$}-isocrystals with geometric origin.
\newblock to appear in Compositio Mathematica.

\bibitem[Kramer-Miller(2020)]{Kramer-Miller3}
Joe Kramer-Miller.
\newblock $p$-adic estimates of abelian {A}rtin {$L$}-functions on curves,
  2020.
\newblock to appear in Forum of Mathematics, $\sigma$.

\bibitem[Kramer-Miller(2021)]{KramerMiller}
Joe Kramer-Miller.
\newblock $p$-adic estimates of exponential sums on curves.
\newblock \emph{Algebra \& Number Theory}, 15\penalty0 (1):\penalty0 141–171,
  Mar 2021.

\bibitem[Kramer-Miller and Upton(2021)]{Kramer-Miller-Upton2}
Joe Kramer-Miller and James Upton.
\newblock Newton polygons of sums on curves {II}: Variation in $p$-adic
  families.
\newblock Preprint, 2021.

\bibitem[Liu and Wei(2007)]{Liu-Wei}
Chunlei Liu and Dasheng Wei.
\newblock The {$L$}-functions of {W}itt coverings.
\newblock \emph{Math. Z.}, 255\penalty0 (1):\penalty0 95--115, 2007.
\newblock ISSN 0025-5874.
\newblock \doi{10.1007/s00209-006-0014-2}.
\newblock URL \url{https://doi.org/10.1007/s00209-006-0014-2}.

\bibitem[Monsky(1968)]{Monsky}
Paul Monsky.
\newblock Formal cohomology {III}.
\newblock \emph{Annals of Mathematics}, 88\penalty0 (2):\penalty0 181--217,
  1968.

\bibitem[Monsky and Washnitzer(1971)]{MW}
Paul Monsky and Gerard Washnitzer.
\newblock Formal cohomology.
\newblock \emph{Annals of Mathematics}, 93\penalty0 (2):\penalty0 315--343,
  1971.

\bibitem[Robba(1984)]{Robba-lower_bounds_NP}
Philippe Robba.
\newblock Index of {$p$}-adic differential operators. {III}. {A}pplication to
  twisted exponential sums.
\newblock Number 119-120, pages 7, 191--266. 1984.
\newblock $p$-adic cohomology.

\bibitem[Sugiyama and Yasuda(2019)]{Sugiyama}
Yusuke Sugiyama and Seidai Yasuda.
\newblock {B}elyi’s theorem in characteristic two.
\newblock \emph{Compositio Mathematica}, 156\penalty0 (2):\penalty0 325–339,
  Dec 2019.
\newblock ISSN 1570-5846.
\newblock \doi{10.1112/s0010437x19007723}.
\newblock URL \url{http://dx.doi.org/10.1112/S0010437X19007723}.

\bibitem[Taguchi and Wan(1997)]{Taguchi}
Yuichiro Taguchi and Daqing Wan.
\newblock Entireness of {$L$}-functions of $\phi$-sheaves on affine complete
  intersections.
\newblock \emph{Journal of Number Theory}, 63\penalty0 (1):\penalty0 170 --
  179, 1997.

\bibitem[Upton(2021)]{Upton2}
James Upton.
\newblock Newton polygons of sums on curves {III}: Overconvergent families.
\newblock Preprint, 2021.

\bibitem[van~der Put(1986)]{Van}
Marius van~der Put.
\newblock The cohomology of {M}onsky and {W}ashnitzer.
\newblock \emph{M\'emoires de la Soci\'et\'e Math\'ematique de France, Nouvelle
  S\'erie}, 23:\penalty0 33--59, 1986.

\bibitem[Wan(1993)]{Wan4}
Daqing Wan.
\newblock Newton polygons of zeta functions and {$L$} functions.
\newblock \emph{Ann. of Math. (2)}, 137\penalty0 (2):\penalty0 249--293, 1993.
\newblock ISSN 0003-486X.
\newblock \doi{10.2307/2946539}.
\newblock URL \url{https://doi.org/10.2307/2946539}.

\bibitem[Wan(2019)]{Wan5}
Daqing Wan.
\newblock Zeta functions of $\mathbb{Z}_p$-towers of curves, 2019.

\end{thebibliography}

\end{document}